\documentclass[11pt, reqno]{amsart}
\usepackage{geometry}                % See geometry.pdf to learn the layout options. There are lots.
\usepackage{graphicx}
\usepackage{amssymb}
\usepackage{epstopdf}
\usepackage{enumitem}
\usepackage{dsfont}
\usepackage{tabu}
\usepackage{bold-extra}
\usepackage{cite}
\usepackage{appendix}
\usepackage{graphicx}
\usepackage{bm}

\newtheorem{theorem}{Theorem}[section]
\newtheorem{corollary}[theorem]{Corollary}
\newtheorem{lemma}[theorem]{Lemma}
\newtheorem{prop}[theorem]{Proposition}

\numberwithin{equation}{section}
\theoremstyle{definition}

\theoremstyle{definition}

\theoremstyle{remark}

\theoremstyle{definition}

\theoremstyle{convention}

\theoremstyle{definition}
\newtheorem*{ack}{Acknowledgments}

\theoremstyle{theorem}

\addtocontents{toc}{\setcounter{tocdepth}{1}}

\title{On the asymmetric additive energy of polynomials}
\author{Oliver McGrath}
\address{Mathematical Institute, University of Oxford, Oxford, OX2 6GG, UK}
\email{oliver.mcgrath@maths.ox.ac.uk}
\begin{document}
\maketitle
%\section{}
%\subsection{}

\begin{abstract}
We prove a general result concerning the paucity of integer points on a certain family of 4-dimensional affine hypersurfaces. As a consequence, we deduce that integer-valued polynomials have small asymmetric additive energy. 
\end{abstract}

\section{Introduction}

Given a non-zero polynomial $f\in \mathbb{Z}[x]$ of degree $d\geq 3,$ an integer $k\in \mathbb{Z},$ and a parameter $B\geq 1,$ we let $E_f(B;k)$ denote the number of integer solutions to the equation
\begin{equation}\label{eq:energy}
f(x_1)+f(x_2)=f(x_3)+f(x_4)+k
\end{equation}
inside the multi-dimensional box
\begin{equation}\label{eq:box}
S(B) = \{ (x_1,x_2,x_3,x_4)\in \mathbb{Z}_{>0}^4: \max_{i} x_i \leq B\}.
\end{equation}
Following Baker, Munsch and Shparlinski~\cite{MuSh}, when $k$ is a fixed, non-zero integer, we call $E_f(B;k)$ the {\it asymmetric} additive energy of the polynomial $f$ inside the box $S(B)$ with respect to $k$. If $k=0$ we simply call $E_f(B;0)$ the {\it symmetric} additive energy of the polynomial $f$ inside the box $S(B).$ 

The latter case has been particularly well-studied in the literature, and here we have results available for any polynomial $f$. In this case, we immediately see there are $2B^2$ diagonal solutions to equation~(\ref{eq:energy}) of the form $(a,b,a,b)$ and $(a,b,b,a)$. Based on standard probabilistic heuristics, one would expect there to be very few other solutions. Indeed, it is now known that one has an asymptotic of the form
\begin{equation}
E_f(B;0) = 2B^2 + O_{f}(B^{2-\delta})
\end{equation}
for some explicit constant $\delta>0$. This was first established in the special case $f(x)=x^d$ by Hooley~\cite{Hooley2,Hooley3}, with results for more general polynomials $f$ being established later by various authors.  We refer the reader to Browning's excellent paper~\cite{Browning} for a brief history of this interesting problem.

The case where $k$ is a fixed, non-zero integer has received less attention. However, recently, it is has been realised that estimates in this alternate setting would have interesting applications. This is the regime we study. In this case, in the absence of any diagonal solutions, one would simply expect there to be very few solutions to equation~(\ref{eq:energy}). The estimate $E_f(B;k) \ll_{f,\epsilon} B^{2+\epsilon}$ is essentially trivial and follows from an application of the divisor bound. Hence in this situation one expects a bound of the form 
\begin{equation}\label{eq:asymp}
E_f(B;k) \ll_{f} B^{2-\delta}
\end{equation}
to hold. We remark that this estimate is uniform in $k$, and this is important in the interest of applications. Currently, a bound like~(\ref{eq:asym}) is only known in the special case when $f(x)=x^d$ and either $d=3$ or $d\geq 5$. This is essentially due to Hooley~\cite{Hooley2} in the former case and Marmon~\cite{Marmon} in the latter. 

In this paper, we prove a bound of type~(\ref{eq:asymp}) holds for an arbitrary polynomial $f$. In other words, we establish that polynomials have small asymmetric additive energy.
\begin{theorem}\label{theo:main}
Fix a polynomial $f\in \mathbb{Z}[x]$ of degree $d \geq 3$. For any non-zero integer $k$ we have
\begin{equation}
E_f(B;k)  \ll_{f} B^{2-1/(50d)}.
\end{equation}
\end{theorem}

%$$
%\delta_d 
%= 
%\begin{cases}
%1/(3d) \,\,&\text{if $3\leq d\leq 4,$} \\
%1-2/\sqrt{d}-1/(d-1)+1/((d-2)\sqrt{d})\,\,&\text{if $5\leq d\leq 19,$} \\
%1/2\,\,&\text{if $d\geq 20$.}
%\end{cases}
%$$

%For ease of presentation, we have not stated the optimal exponent which our method yields. 
%We refer the reader to~Section~\ref{section:proofmethod} for a discussion about the proof of Theorem~\ref{theo:main}.
The zero-set of equation~(\ref{eq:energy}) defines a 4-dimensional affine hypersurface over $\mathbb{Q}$. It is natural to consider this geometric object abstractly. In doing so, we are led to consider a family of 4-dimensional affine hypersurfaces which generalise equation~(\ref{eq:energy}). We are then able to prove the following result concerning the family, from which Theorem~\ref{theo:main} will follow as a corollary. This general result may have independent interest.
\begin{theorem}\label{theo:generalmain}
Fix a polynomial $f\in \mathbb{Z}[x,y]$ of degree $d$ with zero constant term, a polynomial $g\in \mathbb{Z}[x,y]$ of degree $(d-1)$, and non-zero integers $abk\neq 0.$ Write $f_d$ (resp. $g_{d-1}$) for the top-degree homogenous parts of $f$ (resp. $g$), and let $M_{f,g}(B;k)$ denote the number of integer solutions to the equation
\begin{equation}\label{eq:asym}
f(x_1,x_2)=(ax_3-bx_4)g(x_3,x_4)+k
\end{equation}
inside the multi-dimensional box $S(B)$ defined in equation (\ref{eq:box}). Suppose, in addition, the following constraints hold:
\begin{enumerate}
	\item The affine curve $\{f(x,y)=k\}  \subset \mathbb{A}_{\mathbb{Q}}^{2}$ doesn't contain a line.
	\item The projective variety $\{f_{d}(x,y)=0\}  \subset \mathbb{P}_{\mathbb{Q}}^{1}$ is smooth and doesn't contain any non-constant repeated components.
	\item The projective variety $\{(ax-by)g_{d-1}(x,y)=0\} \subset \mathbb{P}_{\mathbb{Q}}^{1}$ doesn't contain any non-constant repeated components.
\end{enumerate}
If $d=4$ we additionally suppose:
\begin{enumerate} 
	\item [(4)] The projective variety 
		$$\bigg\{\frac{1}{a}\frac{\partial g_{d-1}}{\partial x}\bigg(\frac{x}{a},\frac{y}{b}\bigg) +\frac{1}{b}\frac{\partial g_{d-1}}{\partial y}\bigg(\frac{x}{a},\frac{y}{b}\bigg)=0\bigg\} \subset \mathbb{P}_{\mathbb{Q}}^{1}$$
%	$$\bigg\{\frac{1}{a}[\partial_{x}g_{d-1}]\bigg(\frac{x}{a},\frac{y}{b}\bigg) +\frac{1}{b}[\partial_{y} g_{d-1}]\bigg(\frac{x}{a},\frac{y}{b}\bigg)=0\bigg\} \subset \mathbb{P}_{\mathbb{Q}}^{1}$$
	doesn't contain any non-constant repeated components.
\end{enumerate}
Then,
\begin{equation}\label{eq:finalbound}
M_{f,g}(B;k)\ll_{f,g,a,b} B^{2-1/(50d)}.
\end{equation}
\end{theorem}

We remark here that one can show the estimate $M_{f,g}(B;k) \ll_{f,g,a,b,\epsilon} B^{2+\epsilon}$ via a divisor-bound argument, and without constraints on $f$ and $g$ this is essentially optimal. Thus, one may wonder which of the assumptions of Theorem~\ref{theo:generalmain} are necessary in order for there to be significantly fewer solutions. We clearly require (1), as otherwise we would be able to generate $O(B^2)$ ``trivial solutions" lying on lines contained in the hypersurface. Interestingly, (1) is {\it not} sufficient; one also requires (3). This is illustrated by the following example: there are $O(B^{2})$ solutions to the equation
$$x_1^4 - x_2^4 = (x_3-x_4)(x_3^3-x_4^3-3x_4^2-3x_4)+1$$
of the form $(a,a,b+1,b).$ In this example (1) is satisfied (see Lemma~\ref{lemma:nolines} below) but (3) is not, as the top-degree homogenous part of the RHS contains the square-factor $(x_3-x_4)^2.$ On the other hand, (2) and (4) are present purely to facilitate our proof method. Presumably, both of these assumptions could be removed if one had a different approach. We discuss this in more detail in Section~\ref{section:proofmethod}.

We have written the conclusion~(\ref{eq:finalbound}) of Theorem~\ref{theo:generalmain} in the form stated for simplicity; in view of applications, the most important aspect is that we obtain a power saving over the trivial bound. However, our proof actually yields the better bounds:
\begin{equation*}
M_{f,g}(B;k)\ll_{f,g,a,b,\epsilon}
\begin{cases}
B^{2-1/(3d)+\epsilon}\,\,&\text{if $d\in\{3,4\},$} \\
B^{1+\epsilon}(B^{1/2}+B^{2/\sqrt{d}+1/(d-1)-1/((d-2)\sqrt{d})})\,\,&\text{if $d\geq 5.$}
\end{cases}
\end{equation*}
The bounds present in Theorem~\ref{theo:main} can be improved accordingly. We will deduce Theorem~\ref{theo:main} from Theorem~\ref{theo:generalmain} in Section~\ref{subsection:deduce} below. As mentioned above, Theorem~\ref{theo:main} has various applications in the literature. We discuss these now. 

\subsection{Applications} In~\cite{Weyl} Chen, Kerr, Maynard, and Shparlinski were interested in showing that Weyl sums typically exhibit square-root cancellation. A key input to their method was an estimate of the form
$$\sum_{0<|k|\leq 4B^{d}} \frac{E_{f}(B;k)}{k} \ll B^{2-\kappa_d}$$
for the monomial $f(x)=x^d,$ where $\kappa_d>0$ is a constant depending only on $d$ (cf.~ proof of~\cite[Lemma 4.3]{Weyl}). In the paper they establish such an estimate when $d=3$ and $d\geq 5$, using work of Hooley~\cite{Hooley2} and Marmon~\cite{Marmon}. This just left the case $d=4$. It is clear that, with Theorem~\ref{theo:main} applied to the polynomial $f(x)=x^4,$ we can now extend~\cite[Theorem 2.1]{Weyl} to cover the case $d=4$ and hence complete this aspect of the classification of Weyl sums.

\begin{corollary}[Square-root cancellation in Weyl sums almost always]
There exist positive constants $c$ and $C$ such that, for any $d\geq 3$ and any sequence of complex weights $(a_n)_{n=1}^{\infty}$ with $|a_n|=1,$ the set
$$\bigg\{x\in[0,1): cN^{1/2} \leq \bigg|\sum_{n=1}^{N}a_ne^{2\pi i x n^d}\bigg|\leq CN^{1/2} \text{ for  infinitely many } N\in \mathbb{N}\bigg\}$$
has full Lebesgue measure.
\end{corollary}

As a further application of Theorem~\ref{theo:main}, in recent work Baker, Munsch and Shparlinski~\cite{MuSh} proved a general result which enables one to establish large sieve inequalities for a general class of sparse sequences, provided that one has good estimates available for the symmetric and asymmetric additive energy of the sequence. Using the work~\cite{Weyl} described above, the authors were able to establish large sieve inequalities for the monomial sequences $f(x)=x^d$ when $d=3$ or $d\geq 5.$ The authors also proved a weaker result about general polynomial sequences by alternative methods~\cite[Theorem 1.5]{MuSh}. 

By using the new estimates contained in Theorem~\ref{theo:main}, we are able to establish their first result for the monomial $f(x)=x^4$ (albeit with a slightly weaker exponent), and also improve their result concerning polynomial sequences. 
%First note that, from Theorem~\ref{theo:main}, we obtain
%$$\max_{k\neq 0} E_{f}(B;k) \ll_{f,\epsilon} B^{2-\delta_d+\epsilon}.$$
%If $k\geq 5$ this follows directly. If $d\in\{3,4\}$ then whenever $|k|\gg_f B^{2d}$ (say) we have trivially that $E_f(B;k)=0$ from size considerations, and when $0<|k| \ll_f B^{2d}$ we can absorb the $k^{\epsilon}$ factor into the $B^{\epsilon}$ term, by redefining the definition of $\epsilon.$
%
A direct application of~\cite[Theorem 1.1]{MuSh} to the appropriate sequence yields the following.  
\begin{corollary}[Large sieve inequality for polynomial sequence]\label{theo:sieve}
Fix $\epsilon>0.$ For any sequence of complex weights $(a_n)_{n=1}^{\infty}$ and $f\in \mathbb{Z}[x]$ of degree $d\geq 3$ and $Q^d\leq N\leq Q^{2d}$ we have 
$$\sum_{q=1}^{Q}\sum_{\substack{a=1 \\ (a,f(q))=1}}^{f(q)} \bigg| \sum_{n=M+1}^{M+N}a_n e\bigg(\frac{2\pi i a n }{f(q)}\bigg)\bigg|^2 \ll_{f,\epsilon} (NQ^{1/2}+N^{3/4}Q^{d/2+2-1/(50d)})Q^{\epsilon}\sum_{n=M+1}^{M+N}|a_n|^2.$$
\end{corollary}

It is likely that further applications of Theorem~\ref{theo:main} will appear in the literature. We now show how Theorem~\ref{theo:generalmain} implies Theorem~\ref{theo:main}.

\subsection{Deducing Theorem~\ref{theo:main} from Theorem~\ref{theo:generalmain}}\label{subsection:deduce}

Fix a polynomial $p\in \mathbb{Z}[x]$ of degree $d$ and a non-zero integer $k$. Let us write $p(x) = \sum_{i=0}^{d}a_i x^i.$  We would like to apply Theorem~\ref{theo:generalmain}, taking 
\begin{equation}
f(x,y) = (x-y)g(x,y) =p(x)-p(y).
\end{equation}
%\begin{align*}
%f(x,y) &=p(x)-p(y), \\
%g(x,y)&=\frac{p(x)-p(y)}{x-y}.
%\end{align*}
First let us check the constraints on $f$ are satisfied. The homogenous polynomial $f_d(x,y)$ is clearly smooth, and moreover over $\overline{\mathbb{Q}}$ we have the factorisation
\begin{equation}
f_d(x,y)  = a_d(x^d-y^d) = a_d \prod_{\xi^d=1}(x-\xi y),
\end{equation}
which shows that $f_d$ has no repeated factors. Thus we just need to check that the curve $f(x,y)=k$ contains no rational lines. For this we have the following lemma. Note that in this special case we can prove the stronger assertion that $f(x,y)=k$ contains no lines over $\overline{\mathbb{Q}}.$

\begin{lemma}\label{lemma:nolines}
For any $p\in \mathbb{Z}[x]$ and non-zero integer $k$, the affine curve
$$\{p(x)-p(y)=k\} \subset \mathbb{A}_{\overline{\mathbb{Q}}}^{2}$$
contains no lines.
\end{lemma}
\begin{proof}
Suppose for a contradiction that this affine curve contains a line. It is clear that in this case there must exist a parametrisation of this line of the form $(x,y) = (t,\alpha t+ \beta)$ with $\alpha,\beta \in \overline{\mathbb{Q}}$. This leads us to the polynomial identity
\begin{equation}\label{eq:fline}
p(t) = p(\alpha t+\beta)+k
\end{equation}
in $\overline{\mathbb{Q}}[t].$ Comparing leading-term coefficients, we see that $\alpha^d = 1.$ Then comparing the coefficients of $t^{d-1},$ we may solve for $\beta$ and find that $\beta = a_{d-1}(\alpha-1)/(da_d).$ If $\alpha=1$ then we must have $\beta=0,$ and then setting $t=0$ in~(\ref{eq:fline}) yields $k=0$, a contradiction. Otherwise, setting $t = -\beta/(\alpha-1)$ yields the same contradiction.
\end{proof}

It just remains to check our constraints on $g$. It follows from the above that $(x-y)g_{d-1}(x,y)$ is square-free. Finally, when $d=4,$ we must additionally check that the gradient of $g_{d-1}$ is square-free. In this case, $g_3(x,y) = a_4(x^3+x^2y+xy^2+y^3)$ and we have 
\begin{equation}
\frac{\partial g_{3}}{\partial x}(x,y) +\frac{\partial g_{3}}{\partial y}(x,y) = 4a_4(x^2+xy+y^2) = 4a_4(x-\omega y)(x-\overline{\omega}y)
\end{equation}
where $\omega = (-1+\sqrt{3}i)/2.$ Hence the gradient is square-free. This completes the check of all the constraints. Since $E_{p}(B;k) = M_{f,g}(B;k),$ it is now clear that Theorem~\ref{theo:main} follows from Theorem~\ref{theo:generalmain}.

\begin{ack}
The author would like to thank James Maynard for many helpful and insightful discussions about the problem. The author is funded by an EPSRC Studentship and part of Maynard's ERC Grant (grant
agreement No 851318).
\end{ack}

\section{Proof outline of Theorem~\ref{theo:generalmain}}\label{section:proofmethod}

The proof of Theorem~\ref{theo:generalmain} will split into two cases, depending on whether $d\in\{3,4\}$ or $d\geq 5$. In the former case we will apply a sieve method, and in the latter we will apply the determinant method. 

The following notation will be useful in this section and throughout the paper: whenever $F\in \mathbb{Z}[x_1,\ldots,x_n]$ is a polynomial in $n$ variables, we let $M_F(B)$ denote the number of solutions to the equation $F(x_1,\ldots,x_n)=0$ inside the multi-dimensional box $S(B)$ defined in equation~(\ref{eq:box}). (Although this overloads the notation $M_{f,g}(B;k)$ used in Theorem~\ref{theo:generalmain}, it will always be clear from the context which quantity we are referring to.)
 
\subsection{The case $d\in \{3,4\}$ and the polynomial sieve method}

When $d\in\{3,4\}$, we will establish Theorem~\ref{theo:generalmain} via a sieve method. The sieve can be viewed as a ``local" method, where one attempts to rule out the existence of lots of ``global" solutions by ruling out the possibility of lots of ``local" solutions modulo $p$ for ``many" primes $p$. Hooley~\cite{Hooley1,Hooley2,Hooley3} was the first to appreciate how sieves could be applied in this context. We will find it convenient to use a particularly flexible sieve method called the {\it polynomial sieve} due to Browning~\cite{Browning}. We defer the statement of the main sieve proposition to Section~\ref{section:sieve}.

%The main ideas can be traced back to Hooley. 

%The use of sieves to count solutions to Diophantine equations was first realised by Hooley in the 1980s. In a series of papers~\cite{Hooley1,Hooley2,Hooley3}, Hooley developed and utilised a sieve method to obtain bounds for the symmetric additive energy of monomials. Later, building on Hooley's work, Browning~\cite{Browning} developed a flexible and technically simpler sieve method, the so-called {\it polynomial sieve}, and used it to investigate the symmetric additive energy of more general polynomials. We will find it convenient to adapt the polynomial sieve to our problem. 

Let us recall equation~(\ref{eq:asym}) where for simplicitly we assume $a=b=1:$
\begin{equation}
f(x_1,x_2)=(x_3-x_4)g(x_3,x_4)+k
\end{equation}
A key property of the above equation which enables the sieve method to work is the presence of a linear factor on the RHS. If we let $h=(x_3-x_4)$ and eliminate $x_4$ (say) in the above equation, we may equivalently examine 
\begin{equation}
f(x_1,x_2)= hg(x_3,x_3-h)+k
\end{equation}
The argument then proceeds by first fixing the value of $h$, and then counting solutions to the simpler equation in $(x_1,x_2,x_3)$ which remains. We can then apply the polynomial sieve to detect solutions to this equation where the variables $(x_1,x_2)$ are constrained to satisfy the congruence $f(x_1,x_2)\equiv k\,\,(\text{mod}\,\,h).$ By applying the sieve to a congruenced set in this manner, we are able to retain the trivial bound at this step of the argument. This is crucial, as it allows one to obtain a power saving for $M_{f,g}(B;k)$ provided one gains only a small power of $B$ from the sieve estimates. 

This means we would like to sieve by primes $p$ of size $O(B^{\delta})$ (say), which in turn requires one to understand our variables in arithmetic progressions with modulus of size $O(B^{1+2\delta}).$ The modulus is slightly larger than the length of summation, but this difficulty can be overcome by a completion of sums argument. This leaves one with certain exponential sums over algebraic varieties to estimate. Hence, to execute the sieve method effectively, one has recourse to the deep work of Weil~\cite{Weil} and Deligne~\cite{Deligne} concerning the Riemann Hypothesis for curves and higher dimensional varieties over finite fields. This argument works for a generic value of $h$, but in practice there might exist some exceptional values of $h$ for which certain auxiliary varieties (depending on $h$) fail to be smooth. However, using elimination theory, it is possible to show that there can't be too many of these exceptional values. Once can then estimate the contribution from these cases via other methods, such as the Bombieri-Pila method (see below). 

%Browning~\cite{Browning} exploited this fact to great effect, to investigate the symmetric additive energy of an arbitrary quartic polynomial.

\subsection{The case $d\geq 5$ and the determinant method}

For the complementary case, when $d\geq 5,$ we will use the determinant method. The determinant method can be viewed as a ``global" method. The general philosophy is that any ``large" contribution to the count $M_{f,g}(B)$ must come from rational points lying on lower dimensional varieties contained inside the hypersurface. This method has its origins in the pioneering work of Bombieri-Pila~\cite{BombieriPila}. It was then greatly developed at a later date by Heath-Brown~\cite{HeathBrown}, and has enjoyed various refinements since due to a variety of authors. 
%One of the main analytic tools we have available to count rational points on affine (or projective) hypersurfaces is the {\it determinant method.} One can view this as a ``global" method. The determinant method has its origins in the pioneering work of Bombieri-Pila~\cite{BombieriPila} on curves, and was greatly developed at a later date by Heath-Brown~\cite{HeathBrown} to work with higher-dimensional varieties. We will have recourse to both of these works.  The method has subsequently enjoyed various refinements due to a variety of different authors (see~\cite[Chapter 1]{Reuss} and references therein).

Again, recalling equation~(\ref{eq:asym}) with $a=b=1,$ we wish to count points on the 3-dimensional surfaces 
\begin{equation}\label{eq:sec2det}
f(x_1,x_2)=(x_3-n)g(x_3,n)+k \subset \mathbb{A}_{\mathbb{Q}}^{3}
\end{equation}
for each fixed integer $n.$ Thus, as in the sieve method, we begin by considering a simpler object of smaller dimension. However, in the former case it was crucial we had the linear factor on the RHS and our change of variables incorporated this information. This is less important here (however we will still make use of the linear factor later).

We will apply the determinant method to equation~(\ref{eq:sec2det}) in the form of the following result, which is implicit in the proof of~\cite[Theorem 3]{Browning4}. 

\begin{prop}[Browning, Heath-Brown]\label{prop:det}
Let $F\in \mathbb{Z}[x,y,z]$ be a non-singular polynomial of degree $d \geq 4.$ Then
\begin{align*}
M_F(B)&\ll_{d,\epsilon} M_F^{\text{lines}}(B) + B^{1/2+\epsilon}+B^{2/\sqrt{d}+1/(d-1)-1/(\sqrt{d}(d-2))+\epsilon},
\end{align*}
where $M_F^{\text{lines}}(B)$ counts the number of integer points lying on lines contained in the hypersurface $\{F=0\}\subset \mathbb{A}_{\overline{\mathbb{Q}}}^{3}$ inside the box $S(B)$ defined by equation~(\ref{eq:box}).
\end{prop}

We note that the last error term exceeds $B$ when $d<5$. This is the reason we cannot apply the determinant method when $d\in\{3,4\}.$

Therefore, to obtain a power saving for $M_{f,g}(B;k)$ via the determinant method, it is sufficient to have control over the possible lines which can appear in the surfaces~(\ref{eq:sec2det}). Here it is important we are averaging over $n;$ for certain values of $n$ lines may exist and hence contribute a larger amount to $M_{f,g}(B;k),$ however, one can show via elementary means that there cannot be too many values of $n$ for which this can occur.  We note that our argument will also make use of the linear factor on the RHS of~(\ref{eq:sec2det}).

%Our method closely follows the method Browning uses to conclude that the contribution from the off-diagonal terms in the symmetric additive energy case is negligible. 

%We remark here on one aspect of the method which is particularly beneficial for our purposes; namely, the determinant method yields results which are completely uniform with respect to the hypersurface involved, and depend only on the degree and the dimension of the ambient space. This is important for our applications to asymmetric additive energy of polynomials, where we require bounds which are uniform with respect to the constant term $k.$

%It is a feature of the method that one obtains estimates which are completely uniform with respect to the hypersurface involved, and depend only on the degree and the dimension of the ambient space. This is particularly beneficial for our purposes; for applications to the asymmetric additive energy of polynomials, we require bounds which are uniform with respect to the constant term $k.$ This aspect is taken care of by the determinant method, without any additional work.

We remark that Proposition~\ref{prop:det} only applies when the surface~(\ref{eq:sec2det}) is smooth. This will be true for a generic choice of $n$. Thus we arrive at a similar situation to that described above with the sieve argument, where we must handle exceptional cases via different methods. For these values we will use the Bombieri-Pila method. We state the main result we will use here. The following appears as~\cite[Theorem~5]{BombieriPila}.

\begin{prop}[Bombieri-Pila]\label{prop:bombieripila}
Let $F\in \mathbb{Z}[x,y]$ be an absolutely irreducible curve of degree $d\geq 2$. Then 
$$M_F(B) \ll_{d,\epsilon} B^{1/d+\epsilon}.$$
\end{prop}

\subsection{Some basic facts about discriminant polynomials}\label{section:resultants}

Throughout the proof of Theorem~\ref{theo:generalmain} we will encounter various auxiliary curves and surfaces which depend on an integer parameter $h$ (say). For both the determinant method and the polynomial sieve method to work, we require these varieties to be smooth for``most" choices of $h.$ This in turn amounts to showing that certain discriminant polynomials, which by definition will be polynomials in the parameter $h$, are not the zero polynomial. 

Our method of proving this is to extract the leading coefficient using limiting arguments. This leading coefficient will generically be non-zero, and only vanish if our polynomials are degenerate in some way. Our additional assumptions (2) and (4) in the statement of Theorem~\ref{theo:generalmain} ensure that we avoid these cases. It is possible that these additional assumptions could be removed if one had a different way of proving these discriminant polynomials didn't vanish. 
%(As mentioned in the introduction, the assumptons (1) and (3) are necessary to ensure there cannot be many solutions.)

With this in mind, we collect here a few basic facts about discriminant polynomials which we will use without comment throughout the paper. %and can be found in most books on classical algebraic geometry. 
Given a polynomial $f\in \mathbb{Q}[x]$ of degree $d,$ leading coefficient $a_d,$ and roots $\lambda_1,\ldots,\lambda_d \in \overline{\mathbb{Q}},$ we form its discriminant polynomial with respect to $x$, which we write as $\mathrm{Disc}_x[f(x)]$ by the formula
\begin{equation}
\mathrm{Disc}_x[f(x)] = (-1)^{d(d-1)/2}a_d^{d(d-1)} \prod_{\substack{i\neq j}}(\lambda_i-\lambda_j).
\end{equation}
%We will often refer to the discriminant of $f$, denoted $\mathrm{Disc}_xf(x)$. The two notions are closely related by the formula
%\begin{equation}
%\mathrm{Res}_x(f(x),f'(x)) = (-1)^{d(d-1)/2} a_d\mathrm{Disc}_xf(x).
%\end{equation}
The discriminant polynomial satisfies the following properties:
\begin{enumerate}[label=(\roman*)]
	\item $\mathrm{Disc}_x[f(x)]$ is a polynomial in the coefficients of $f$.
	\item $\mathrm{Disc}_x[f(x)]$ vanishes if and only if $f$ and $f'$ possess a common factor over $\mathbb{Q}$. In particular, if $f$ has no non-constant repeated factors over $\mathbb{Q}$ then $\mathrm{Disc}_x[f(x)]\neq 0.$
	\item For any real numbers $a,b$ and $c$ we have the transformation formula 
	\begin{equation}
	\mathrm{Disc}_x[af(bx+c)] = a^{2d-2}b^{d(d-1)}\mathrm{Disc}_x[f(x)].
	\end{equation}
\end{enumerate}
%We also isolate the following metric property of resultant polynomials which follows readily from the properties listed above.
%\begin{enumerate}
%	\item[(iv)] Suppose $f_n(x),g_n(x)\in \mathbb{Q}[x]$ for each integer $n\geq 1,$ and we have the point-wise convergence $f_n \rightarrow f, g_n \rightarrow g$ where $f,g\in \mathbb{Q}[x]$ and moreover  $\mathrm{deg}(f_n) = \mathrm{deg}(f)$ and $\mathrm{deg}(g_n) = \mathrm{deg}(g)$ for every $n.$ Then we have the pointwise convergence
%\begin{equation}
%\mathrm{Res}_x(f_n(x),g_n(x)) \rightarrow \mathrm{Res}_x(f(x),g(x)).
%\end{equation}
%\end{enumerate} 

\section{Notation}\label{section:notation}

We will use both Landau and Vinogradov asymptotic notation throughout the paper. $B$ will denote a large integer, and all asymptotic notation is to be understood as referring to the limit as $B\rightarrow \infty.$ We allow any implied constants to depend implicitly on the variables $f,g,a$ and $b,$ without specifying so. By this, we mean dependencies on the coefficients of $f$ and $g$ and also on the degree $d$. Any dependencies of the implied constants on other parameters $A$ will be denoted by a subscript, for example $X\ll_{A} Y$ or $X=O_{A}(Y),$ unless stated otherwise. We let $\epsilon$ denote a small positive constant, and we adopt the convention it is allowed to change at each occurrence, and even within a line.
% and we will explicitly state any dependencies in implied constants where appropriate. 

If $k$ is a field, we let $\mathbb{A}_{k}^{n}$ (resp. $\mathbb{P}_{k}^{n}$) denote $n$-dimensional affine (resp. projective) space over $k$. If $F\in \mathbb{Z}[x_1,\ldots,x_n]$ we let $\{F=0\}\subset \mathbb{A}_{\mathbb{Q}}^{n}$ denote the $n$-dimensional affine hypersurface generated by $F$ over $\mathbb{Q}.$ By slight abuse of notation, we may write this as $F=0\subset \mathbb{A}_{\mathbb{Q}}^{n},$ or even simply $F\subset \mathbb{A}_{\mathbb{Q}}^{n}.$ We adopt similar conventions whenever $F$ is homogenous and generates a projective hypersurface. The hypersurface defined by $F$ is said to be {\it smooth} over $\mathbb{Q}$ if the system of equations 
$$F(y_1,\ldots,y_n) = \frac{\partial{F}}{\partial{x_1}}(y_1,\ldots,y_n) =\ldots =\frac{\partial{F}}{\partial{x_d}}(y_1,\ldots,y_n)= 0$$
has no solutions with $(y_1,\ldots,y_n)\in\mathbb{A}_{\mathbb{Q}}^{n}$ in the affine case or $[y_1:\ldots:y_n]\in\mathbb{P}_{\mathbb{Q}}^{n}$ in the projective case.  
%The hypersurface defined by $F$ is said to be {\it absolutely irreducible} over $\mathbb{Q}$ if it is irreducible over $\overline{\mathbb{Q}}.$

We let $f_i$ (resp. $g_i$) denote the homogenous part of $f$ (resp. $g$) of degree $i$. Thus, we may write 
$$f(x,y) = \sum_{i=0}^{d}f_i(x,y),\,\,\,\,g(x,y) = \sum_{i=0}^{d-1}g_i(x,y).$$
We will frequently use assumption (1) in Theorem~\ref{theo:generalmain} which says that the curve  $\{f(x,y)=k\}\subset \mathbb{A}_{\mathbb{Q}}^{2}$ contains no lines. We note that the curve may well contain lines over the larger field $\overline{\mathbb{Q}},$ but these lines cannot simply be reparametrisations of lines over $\mathbb{Q}$ (e.g. $\sqrt{2}x+\sqrt{2}y=\sqrt{2}$). To this end, the following definition is useful: we say a line $\{\alpha x+\beta y + \gamma = 0\} \subset \mathbb{A}_{\overline{\mathbb{Q}}}^{2}$ is {\it definable over} $\mathbb{Q}$ if there exists $\lambda \in \overline{\mathbb{Q}}$ and $a,b,c\in \mathbb{Q}$ such that $\alpha x+ \beta y+\gamma = \lambda (ax+by+c).$  With this definition, our assumption is precisely that the curve $f(x,y)=k$ doesn't contain any lines definable over $\mathbb{Q}.$

% By this, we mean the curve doesn't contain any lines {\it colinear} to a line with rational coefficients; however, it may well contain lines over the larger field $\overline{\mathbb{Q}}.$
%For example, the curve may well contain lines over the larger field $\overline{\mathbb{Q}}$, but this line cannot be something like $\sqrt{2}x+\sqrt{2}y+\sqrt{2}=0$.
%We say a line $\{\alpha x+\beta y + \gamma = 0\} \subset \mathbb{A}_{\overline{\mathbb{Q}}}^{n}$ is colinear to a rational line if there exists $\lambda \in \overline{\mathbb{Q}}$ and $a,b,c\in \mathbb{Q}$ such that $\alpha x+ \beta y+\gamma = \lambda (ax+by+c).$

%%%%%%%%%%%
%%%%%%%%%%%
%%%%%%%%%%%
%%%%%%%%%%%
%%%%%%%%%%%

\section{Proof of Theorem~\ref{theo:generalmain} in the case $d\geq 5.$} 

Fix $f,g\in\mathbb{Z}[x,y]$ as in the statement of Theorem~\ref{theo:generalmain}, with $d\geq 5$. In this section we prove the Theorem~\ref{theo:generalmain} in this regime. Recalling equation~(\ref{eq:asym}), we wish to count integer points on the affine hypersurface
\begin{equation}
f(x_1,x_2)=(ax_3-bx_4)g(x_3,x_4)+k\subset \mathbb{A}_{\mathbb{Q}}^{4}.
\end{equation}
The method proceeds by fixing the value of $x_4$, which we now call $n$, and considering the resulting 3-dimensional affine surface 
\begin{equation}
f(x_1,x_2) = (ax_3-bn)g(x_3,n)+k\subset \mathbb{A}_{\mathbb{Q}}^{3},
\end{equation}
which we call $\Gamma_n.$ For later purposes, we let $\Gamma_n^{\text{proj}}(x_1,x_2,x_3,w)\subset \mathbb{P}_{\mathbb{Q}}^{3}$ denote the projectivisation of this surface. Recalling our notation so far, we can write 
\begin{equation}\label{eq:gammaneq}
M_{f,g}(B;k) = \sum_{1\leq n\leq B} M_{\Gamma_n}(B).
\end{equation}
First let us deal with a degenerate situation, where $n$ is such that $g(x_3,n)$ vanishes identically (as a polynomial in $x_3$). 
\begin{lemma}\label{lemma:splitcomponents}
Suppose $n$ is such that $g(x_3,n)$ vanishes identically. Then 
$$M_{\Gamma_n}(B) \ll_{\epsilon} B^{3/2+\epsilon}.$$
\end{lemma}
\begin{proof}
Clearly 
\begin{equation}
M_{\Gamma_n}(B) \ll B\cdot \#\{x_1,x_2\in[1,B]\cap\mathbb{Z}: f(x_1,x_2) = k\}.
\end{equation}
To evaluate this count, we can decompose our curve into $O(1)$ absolutely irreducible components and majorise by summing over each component. We can then apply Proposition~\ref{prop:bombieripila} to each component. Components of degree $\geq 2$ contribute $O_{\epsilon}(B^{1/2+\epsilon}).$ By assumption, any components of degree 1 (i.e. lines) are not definable over $\mathbb{Q}$ and therefore contain at most 1 integer point, and so these contribute in total $O(1)$. Thus, this count is $O_{\epsilon}(B^{1/2+\epsilon}),$ which yields the lemma. 
\end{proof}

There can be at most $O(1)$ such values of $n$ for which $g(x_3,n)$ vanishes identically, and so~(\ref{eq:gammaneq}) becomes
\begin{equation}\label{eq:gammaneq2}
M_{f,g}(B;k) = \sum_{\substack{1\leq n\leq B \\ g(\cdot,n)\neq 0}}M_{\Gamma_{n}}(B)+O_{\epsilon}(B^{3/2+\epsilon}).
\end{equation}
For these remaining values of $n$ we would like to use Proposition~\ref{prop:det} to estimate the corresponding $M_{\Gamma_n}(B).$ To do this we require the surface $\Gamma_n$ to be smooth. Generically this is will be the case, as the following lemma demonstrates. 

\begin{lemma}\label{lemma:singvals}
$\Gamma_{n}^{\text{proj}}$ is singular for at most $O(1)$ values of $n$. 
\end{lemma}
\begin{proof}
Write $g(x_3,x_4) = \sum_{i+j\leq d-1}e_{i,j} x_3^ix_4^j.$ We first deal with possible singular points with $w=0.$ We can write $\Gamma_n^{\text{proj}}$ as
\begin{align*} 
f_d(x_1,x_2)+wf_{d-1}(x_1,x_2) &=  ae_{d-1,0}x_3^d+(ae_{d-2,1}n-bne_{d-1,0}+ae_{d-2,0})x_3^{d-1}w\\
&+\text{ (terms involving $w^2$)}.
\end{align*}
If $[r:s:t:0]$ is a singular point, by considering the $x_1$ and $x_2$ derivatives, we see that necessarily 
$$\frac{\partial f_{d}}{\partial x_1}(r,s) = \frac{\partial f_{d}}{\partial x_2}(r,s) = 0.$$
By Euler's identity, we see that $f_d(r,s)=0$ also. As we are assuming $f_d$ is smooth we must have $r=s=0.$ If $e_{d-1,0}\neq 0$ then the equation above yields $t=0,$ which is a contradiction. If $e_{d-1,0}=0$ then necessarily $e_{d-2,1}\neq 0$ as otherwise $g_{d-1}(x_3,x_4)$ would contain a square factor of $x_4^2$. In this case the $w$ derivative evaluated at $[0:0:t:0]$ yields 
$$a(e_{d-2,1}n+e_{d-2,0})t^{d-1} = 0.$$
Unless $n=-e_{d-2,0}/e_{d-2,1}$ we conclude again that $t=0.$ Thus there is at most 1 value of $n$ for which $\Gamma_n^{\text{proj}}$ contains a singular point with $w=0.$

This leaves us to examine possible singular points with $w=1.$ For ease, let us write $G_{n}(x_3) = (ax_3-bn)g(x_3,n)$ and denote by $G_n'(x_3)$ the derivative with respect to $x_3.$ Again, by examining derivatives, it is clear that any singular point $[r:s:t:1]$ must satisfy, in particular, 
$$\frac{\partial f}{\partial x_1}(r,s) = \frac{\partial f}{\partial x_2}(r,s) = 0 \text{ and } G_n'(t) = 0.$$
Now, our assumption that $f_d$ is square-free implies that the partial derivatives $\partial f/\partial x_1$ and $\partial f/\partial x_2$ are coprime and hence, by B\'ezout's theorem, they have at most $O(1)$ common zeros.\footnote{If $\partial f/\partial x_1$ and $\partial f/\partial x_2$ have a common factor $p$ then $f$ must be of the form $f= p^2q+c$ for some constant $c$. By comparing homogenous parts of top-degree we see $f_d$ will be divisible by a square in this case.} Thus, this system constrains $r$ and $s$ to at most $O(1)$ possible values. For each pair we must then solve the system
\begin{align*}
G_{n}(t) &= f(r,s), \\
G_n'(t) &= 0.
\end{align*}
We would like to show that this system is only solvable in $t$ for at most $O(1)$ choices of $n.$ If this were the case, it would follow that there are at most $O(1)$ values of $n$ for which $\Gamma_n^{\text{proj}}$ contains a singular point with $w=1.$  This, together with the above, would yield the conclusion of the lemma. 

We will prove this via the following strategy, which will be used numerous times throughout the paper: if the system is solvable then the discriminant
\begin{equation*}
\mathrm{Disc}_{t} [G_{n}(t) -f(r,s)]
\end{equation*}
will vanish identically. However, by definition, this will be a polynomial in $n$ which generically will be non-zero and so only vanish for $O(1)$ values of $n$. Hence it is sufficient to prove that this discriminant is not identically zero. We prove this by extracting the leading coefficient. Our assumptions on $f$ and $g$ will then imply this leading coefficient doesn't vanish. 

Now we have
\begin{align} \nonumber
\frac{G_{n}(n t)}{n^d} &= \frac{(at-b)g(nt,n)}{n^{d-1}} \\ \nonumber
&= \frac{(at-b)}{n^{d-1}}  \sum_{i=0}^{d-1}g_i(nt,n) \\ \nonumber
&= (at-b)g_{d-1}(t,1)+O(n^{-1}).
\end{align}
Thus, taking limits, we see that
\begin{align}\label{eq:limit} 
\lim_{n\rightarrow \infty} \frac{G_{n}(n t)}{n^d} &=(at-b)g_{d-1}(t,1).
\end{align}
By standard properties of discriminant polynomials, as detailed in Section~\ref{section:resultants}, whenever $n\neq 0$ we can write 
\begin{align*}
\mathrm{Disc}_{t} \bigg[\frac{G_{n}(nt) -f(r,s)}{n^d} \bigg] &= \frac{\mathrm{Disc}_{t}[G_n(t) - f(r,s)]}{n^{(D-1)(2d-D)}},
\end{align*}
where $1\leq D\leq d$ is the degree of $G_n(t)$ as a polynomial in $t.$ This is valid for any $n\neq 0,$ and moreover both sides of equation (\ref{eq:limit}) have the same degree $D$ in $t$. Thus we may take the limit as $n\rightarrow \infty$ inside the discriminant, and it follows that 
$$\lim_{n\rightarrow \infty} \frac{\mathrm{Disc}_{t}[G_n(t) - f(r,s)]}{n^{(D-1)(2d-D)}} = \mathrm{Disc}_{t} [(at-b)g_{d-1}(t,1)]$$
The RHS is non-zero by our assumption that $\{(ax-by)g_{d-1}(x,y)\}\subset \mathbb{P}_{\mathbb{Q}}^{2}$ is square-free. Thus, the discriminant polynomial has a non-zero leading coefficient. This completes the proof of the lemma.
\end{proof}

Before we can dispense with those values of $n$ for which $\Gamma_n$ is singular, we require some information about the possible lines which can be contained in level sets of the curve  $\{f(x,y)=0\}\subset \mathbb{A}_{\overline{\mathbb{Q}}}^{2}.$ 

\begin{lemma}[Lines contained in level sets of $f$]\label{lemma:lines}
Let $f\in \mathbb{Z}[x_1,x_2]$ be such that $f_d$ is not divisible by a square and let $l\in \overline{\mathbb{Q}}.$ Then, if the variety 
$$\{f(x_1,x_2) = l\} \subset \mathbb{A}_{\overline{\mathbb{Q}}}^{2}$$
contains a line, this line must be equal to one of the possible lines listed below:
\begin{enumerate}
	\item The line parametrised by $(x_1,x_2) = (t,\alpha t+\beta)$ where $f_d(1,\alpha) = 0$ and $\beta = -f_{d-1}(1,\alpha)/(\partial f_d/\partial y)(1,\alpha).$ 
	\item The line parametrised by $(x_1,x_2) = (\gamma,t)$ where $\gamma=-f_{d-1}(0,1)/(\partial f_d/\partial x)(0,1).$ (This case requires $f_d(0,1)=0.$)
\end{enumerate}
\end{lemma}
\begin{proof}
Let us suppose that the level set $f(x_1,x_2)=l$ contains a line. This line can be parametrised by $(x_1,x_2) = (\lambda_1 t+\mu_1,\lambda_2 t+ \mu_2),$ where all coefficients lie in $\overline{\mathbb{Q}}$ and $\lambda_1,\lambda_2$ are not both zero. We have the Taylor expansion
\begin{align}\label{eq:taylorexpansion}
f(\lambda_1 t+\mu_1,\lambda_2 t+\mu_2)  &=\sum_{i=0}^{d} \sum_{j=0}^{i} \bigg[\sum_{m+k=j} \frac{1}{k!} \frac{1}{m!}\frac{\partial^{j}f_i}{\partial x^ky^m}(\lambda_1 ,\lambda_2 )\mu_1^k\mu_2^m\bigg] t^{i-j}.
\end{align}
We consider two cases, depending on whether or not $\lambda_1$ is zero.

\begin{enumerate}
	\item If $\lambda_1 \neq 0$ then our line may be parametrised by $(x_1,x_2) = (t,\alpha t+\beta)$ for some $\alpha$ and $\beta.$ Now since we are assuming $f(t,\alpha t+\beta)=l,$ we arrive at the following polynomial identity in $\overline{\mathbb{Q}}[t]:$ 
\begin{align*}
f_d(1,\alpha) t^{d}+\bigg[f_{d-1}(1,\alpha)+\beta\frac{\partial f_d}{\partial y}(1,\alpha)\bigg]t^{d-1}+\ldots=l.
\end{align*}
Since $f_d$ is smooth, $f_d(1,\alpha)$ is a non-zero polynomial in $\alpha$ of degree at most $d.$ Therefore, there are at most $O_d(1)$ choices of $\alpha$ for which the leading coefficient vanishes. For every such $\alpha,$ we must have $(\partial f_d/\partial y)(1,\alpha)\neq 0$ as otherwise the discriminant $\mathrm{Disc}_y(f_d(1,y))$ would vanish, contradicting the fact $f_d$ is square-free. The result follows by looking at the vanishing of the coefficient of $t^{d-1}$.
	\item If $\lambda_1 = 0$ then necessarily $\lambda_2 \neq 0.$ In this case our line may be parametrised by $(x_1,x_2) = (\gamma,t).$ Now 
\begin{align*}
f(\gamma,t)  = f_d(0,1)t^d+\bigg[f_{d-1}(0,1)+\gamma\frac{\partial f_d}{\partial x}(0,1)\bigg]t^{d-1}+\ldots.
\end{align*}
For the coefficient of $t^d$ to vanish we must have $f_d(0,1)=0.$ It follows that $(\partial f_d/\partial x)(0,1)\neq 0$ as otherwise $f_d(x_1,x_2)$ would be divisible by the square $x_1^2.$ The coefficient of $t^{d-1}$ vanishing then implies $\gamma = -f_{d-1}(0,1)/(\partial f_d/\partial x)(0,1),$ as required.
\end{enumerate}
\end{proof}

From now on we let $\Lambda \subset \overline{\mathbb{Q}}$ consist of the set of $\alpha,\beta,\gamma$ defined in Lemma~\ref{lemma:lines} above, whenever they exist. These numbers depend only on $f_d$ and $f_{d-1}$, and it is clear that $|\Lambda|=O(1).$ In case (1) we must have $l=f(0,\beta)$ and in case (2) we must have $l=f(\gamma,0).$

We are now finally in a position to deal with the contribution from those $n$ for which $\Gamma_n^{\text{proj}}$ is singular. 

\begin{lemma}\label{lemma:dealwithsingvals}
We have 
$$\sum_{\substack{1\leq n\leq B \\ \Gamma_n^{\text{proj}}\text{ singular} \\ g(\cdot,n)\neq 0}}M_{\Gamma_n}(B) \ll_{\epsilon} B^{3/2+\epsilon}.$$
\end{lemma}
\begin{proof}
\sloppy There are $O(1)$ choices of $x_3$ for which $(ax_3-bn)g(x_3,n)=(l-k)$ and $l\in \{f(0,\beta),f(\gamma,0)\}_{\beta,\gamma\in \Lambda}.$ For these values of $x_3$ we use the trivial bound $O(B)$ for the number of possible values of $x_1,x_2\in[1,B]\cap\mathbb{Z}$ for which $f(x_1,x_2) = l$. For the other $O(B)$ values of $x_3,$ we claim that
$$\#\{x_1,x_2\in[1,B]\cap\mathbb{Z}: f(x_1,x_2) = (ax_3-bn)g(x_3,n) +k \} \ll_{\epsilon} B^{1/2+\epsilon}.$$
%This follows from Proposition~\ref{prop:bombieripila}, by splitting our curve into absolutely irreducible components and summing over each component. By Lemma~\ref{lemma:lines} we are avoiding any level set which could potentially contain a line over $\overline{\mathbb{Q}},$ and hence every absolutely irreducible component of our curve must have degree $\geq 2.$ 
Indeed, this follows from Proposition~\ref{prop:bombieripila} in much the same way as Lemma~\ref{lemma:splitcomponents}. We split our curve into absolutely irreducible components and sum over each component. By Lemma~\ref{lemma:lines} we are avoiding any level set which could potentially contain a line over $\overline{\mathbb{Q}},$ and hence every absolutely irreducible component of our curve must have degree $\geq 2.$ By Proposition~\ref{prop:bombieripila} we can therefore bound this count by $O_{\epsilon}(B^{1/2+\epsilon}),$ as required. We are done as there are only $O(1)$ choices for $n$ by Lemma~\ref{lemma:singvals} and only $O(1)$ choices for $\beta,\gamma$ by Lemma~\ref{lemma:lines}.
\end{proof}
Lemma~\ref{lemma:dealwithsingvals} together with equation~(\ref{eq:gammaneq2}) yields
\begin{equation}
M_{f,g}(B;k) = \sum_{\substack{1\leq n\leq B \\  \Gamma_n^{\text{proj}}\text{ smooth} \\ g(\cdot,n)\neq 0}}M_{\Gamma_{n}}(B)+O_{\epsilon}(B^{3/2+\epsilon}).
\end{equation}
We are now in a position to apply Proposition~\ref{prop:det} to estimate each term in the sum. Let us analogously define $M_{\Gamma_n}^{\text{lines}}(B)$ to count the number of integer points lying on a line contained in the surface $\{\Gamma_n=0\}\subset \mathbb{A}_{\mathbb{Q}}^{3}.$ From Proposition~\ref{prop:det}, we conclude that 
\begin{equation}\label{eq:gammaneq3}
M_{f,g}(B;k) = \sum_{\substack{1\leq n\leq B \\  \Gamma_n^{\text{proj}}\text{ smooth} \\ g(\cdot,n)\neq 0}}M_{\Gamma_n}^{\text{lines}}(B)+O_{\epsilon}(B^{1+\epsilon}(B^{1/2}+B^{2/\sqrt{d}+1/(d-1)-1/((d-2)\sqrt{d})})).
\end{equation}
We turn to understanding the lines which can appear in $\Gamma_n.$ The reason we work projectively is so that we can apply the following lemma due to Colliot-Th\'el\`ene, which can be found in~\cite[Appendix]{HeathBrown}. 

\begin{lemma}[Colliot-Th\'el\`ene]\label{lemma:col}
Suppose that $X\subset \mathbb{P}_{\mathbb{Q}}^{3}$ is a smooth projective surface of degree $d\geq 3.$ Then there are $O_d(1)$ lines contained in $X$. 
\end{lemma} 
%This generalises the fact that any smooth cubic surface contains 27 lines.
%We have the following. Recall, we write $f(x_1,x_2) =\sum_{i+j\leq d} c_{i,j}x_1^ix_2^j.$ 
The following proposition, reminiscent of Lemma~\ref{lemma:lines} above, summarises our information about possible lines contained in the affine surfaces $\Gamma_n.$ 
\begin{prop}[Analysis of lines contained in $\Gamma_n$]\label{prop:lines}
Suppose $f$ and $g$ satisfy the hypotheses of Theorem~\ref{theo:generalmain} and let $\Lambda$ be defined as in the remarks following Lemma~\ref{lemma:lines}. Fix a positive integer $n$ for which $g(x_3,n)$ is not identically zero as a polynomial in $x_3.$ Then, if the variety 
$$\{\Gamma_n=0\} \subset \mathbb{A}_{\overline{\mathbb{Q}}}^{3}$$
contains a line, this line must be equal to one of the possible lines listed below:
\begin{enumerate}
	\item The line $x_2 = \alpha x_1+\beta,$ where $\alpha,\beta \in \Lambda$ and $x_3$ and $n$ satisfy the equation $(ax_3-bn)g(x_3,n)+k=f(0,\beta)$ with $f(0,\beta)\neq k.$
	\item The line $x_1 = \gamma$, where $\gamma\in\Lambda$ and $x_3$ and $n$ satisfy $(ax_3-bn)g(x_3,n)+k= f(\gamma,0)$ where $f(\gamma,0)\neq k.$ (This case requires $f_d(0,1)=0.$)
	\item Lines which contain at most 1 integer point $(x_1,x_2,x_3).$
\end{enumerate}
\end{prop}
\begin{proof}
Any line in the surface can be parametrised by $x_i = \lambda_i t+\mu_i$ for $i\in\{1,2,3\},$ where $\lambda_i,\mu_i \in \overline{\mathbb{Q}}$ and the $\lambda_i$ are not all zero. This then leads to an equality of polynomials in $\overline{\mathbb{Q}}[t]:$ 
$$f(\lambda_1 t+ \mu_1,\lambda_2 t+ \mu_2) = [a(\lambda_3 t+ \mu_3)-bn]g(\lambda_3 t+ \mu_3,n)+k.$$
Our proof proceeds by a careful case analysis.
\begin{enumerate}
	\item Suppose that $\lambda_3=0$ and $\lambda_1\neq 0.$ Then, by Lemma~\ref{lemma:lines}, our line must be equal to the line with parametrisation 
$$x_1 = t,\,\,\, x_2 = \alpha t+ \beta,\,\,\,x_3 =  \mu_3$$
where $\alpha,\beta \in \Lambda$ and $x_3=\mu_3$ must satisfy $(ax_3-bn)g(x_3,n)+k=f(0,\beta).$ If $f(0,\beta)=k$ then, because we are assuming the curve $f(x,y)=k$ doesn't contain a line definable over $\mathbb{Q},$ we must have at least one of $\alpha,\beta\in \overline{\mathbb{Q}}\backslash \mathbb{Q}.$ But now the line $x_2 = \alpha x_1+\beta$ contains at most one integer point $(x_1,x_2).$
	\item Suppose that $\lambda_3=0$ and $\lambda_1= 0.$ Then, by Lemma~\ref{lemma:lines}, our line must equal the line with parametrisation
$$x_1 = \gamma,\,\,\,x_2 = t,\,\,\,x_3 =  \mu_3$$
where $\gamma\in \Lambda$ and we see $x_3=\mu_3$ must satisfy 
	\begin{align*}
	f(\gamma,0) &= (ax_3-bn)g(x_3,n)+k.
	\end{align*}
This case requires $f_d(0,1)=0.$ From the definition of $\gamma$ in Lemma~\ref{lemma:lines}, we see that $\gamma \in \mathbb{Q}.$ We are assuming that $f(x,y)=k$ doesn't contain any lines definable over $\mathbb{Q}$, and so it follows that in this situation we must have $f(\gamma,0)\neq k.$\footnote{For otherwise we would have $f(\gamma,t)=k$ identically in $t,$ and then $f(x,y)=k$ would contain the rational line $(x-\gamma)$.}
	\item Suppose that $\lambda_3\neq 0.$ We may parametrise our line as follows:
$$x_1 = \tilde{\lambda}_1 t+ \tilde{\mu}_1,\,\,\,x_2 = \tilde{\lambda}_2 t+ \tilde{\mu}_2,\,\,\,x_3 = t,$$
where $\tilde{\lambda}_1,\tilde{\mu}_1,\tilde{\lambda}_2,\tilde{\mu}_2\in\overline{\mathbb{Q}}.$ Then we must examine 
	$$f(\tilde{\lambda}_1 t+ \tilde{\mu}_1,\tilde{\lambda}_2 t+ \tilde{\mu}_2) = (at-bn)g(t,n)+k.$$	
	Recall we are supposing that $g(t,n)$ is not the zero polynomial. In particular, as polynomials in $t,$ we must have the factorisation
$$at-bn | f(\tilde{\lambda}_1 t+ \tilde{\mu}_1,\tilde{\lambda}_2 t+ \tilde{\mu}_2)-k$$
	over $\overline{\mathbb{Q}}[t].$ Now, because we are assuming that the curve $f(x,y)=k$ doesn't contain a line definable over $\mathbb{Q}$, it follows that at least one of the variables $\tilde{\lambda}_1,\tilde{\mu}_1,\tilde{\lambda}_2,\tilde{\mu}_2\in \overline{\mathbb{Q}}\backslash\mathbb{Q}.$ But then, since $x_i = \tilde{\lambda}_i x_3 + \tilde{\mu}_i$ for $i\in\{1,2\},$ it is clear that any line which arises in this way contains at most $1$ integer point.
\end{enumerate}
\end{proof}

Our last technical estimate is the following lemma. We note that our assumption $(ax-by)g(x,y)$ is square-free means that, in particular, we have $g_{d-1}(1,a/b)\neq 0.$

\begin{lemma}\label{lemma:g}
Suppose $g\in\mathbb{Z}[x,y]$ is such that $g_{d-1}(1,a/b)\neq 0.$ Then for any $l\neq 0$ we have
$$\#\{x,y\in[1,B]\cap\mathbb{Z}: (ax-by)g(x,y)=l\}\ll_{\epsilon} B^{1/2+\epsilon}.$$ 
\end{lemma}
\begin{proof}
This is proved along the same lines of Lemma~\ref{lemma:splitcomponents}. We will be done by Proposition~\ref{prop:bombieripila} provided that we can show this curve doesn't contain any lines defined over $\overline{\mathbb{Q}}.$ As $l\neq 0,$ any line contained in this variety must have a parametrisation of the form $(x,y)=(t,\lambda t+ \mu).$ In this case we must have the following polynomial identity in $\overline{\mathbb{Q}}[t]:$
$$[(a-b\lambda)t-b\mu]g(t,\lambda t +\mu)=l.$$
Since $l\neq 0,$ for this to be true clearly the first factor must be constant, i.e. $\lambda = a/b.$ Now, by Taylor expansion, we have
$$g(t,\lambda t +\mu) = g_{d-1}(1,\lambda)t^{d-1}+\ldots.$$
For the leading term to vanish we must have $g_{d-1}(1,\lambda) = g_{d-1}(1,a/b) = 0.$ This is a contradiction.
\end{proof}

The following lemma, together with equation~(\ref{eq:gammaneq3}), completes the proof of Theorem~\ref{theo:generalmain}, in the case $d\geq 5.$

\begin{lemma}\label{lemma:gammanfinal}
We have
$$ \sum_{\substack{1\leq n\leq B \\  \Gamma_n^{\text{proj}}\text{ smooth} \\ g(\cdot,n)\neq 0}}M_{\Gamma_n}^{\text{lines}}(B) \ll_{\epsilon} B^{3/2+\epsilon}$$
\end{lemma}
\begin{proof}
This follows by assembling the information gathered thus far. Consider again Proposition~\ref{prop:lines}. It is easy to see that the contribution from lines in cases (1) or (2) is $O_{\epsilon}(B^{3/2+\epsilon})$ by Lemma~\ref{lemma:g} together with the fact $|\Lambda| = O(1),$ and the contribution from lines in case (3) is $O(B)$ by Lemma~\ref{lemma:col}.
\end{proof}
%\footnote{Suppose for a contradiction there exists two distinct values of $x_1$, say $x_1\neq \tilde{x}_1,$ for which $\alpha x_1+\beta= m$ and $\alpha \tilde{x}_1+\beta = n$ with $m,n\in \mathbb{Z}.$ Then we must have $\alpha(x_1-\tilde{x}_1) = m-n \in \mathbb{Z}.$ Since $(x_1-\tilde{x}_1)\in \mathbb{Z}\backslash\{0\}$ we must then have $\alpha \in \mathbb{Q}.$ Now $\beta = \frac{f_{d-1}(1,\alpha)}{(\partial_y f_d)(1,\alpha)}$ and so we see that in this case $\beta \in \mathbb{Q}.$ But then $\alpha,\beta \in \mathbb{Q}$ and so $f(0,\beta)\neq 0,$ a contradiction.}
We conclude that 
\begin{equation}
M_{f,g}(B;k) \ll_{\epsilon} B^{1+\epsilon}(B^{1/2}+B^{2/\sqrt{d}+1/(d-1)-1/((d-2)\sqrt{d})}).
\end{equation}
One can check the exponents appearing here are strictly less than $(2-1/(50d))$ whenever $d\geq 5$, and so the bound stated in Theorem~\ref{theo:generalmain} follows.

\section{Proof of Theorem~\ref{theo:generalmain} in the case $d\in\{3,4\}.$}\label{section:sieve}

We now proceed to prove Theorem~\ref{theo:generalmain} in the remaining cases when $d\in\{3,4\}.$ As discussed in Section~\ref{section:proofmethod}, we will use the polynomial sieve developed by Browning in this regime. We state the main sieve proposition here. By slightly adjusting the set-up, we are able to make the implied constant absolute and transfer any dependencies into our choice of $\mathcal{P},$ the set of sieving primes. This is a technical convenience which will prove useful to us, as for our applications to asymmetric additive energy of polynomials we wish to explicitly keep track of any dependencies on the constant term. 

It is clear that the following result follows from the proof of~\cite[Theorem 1.1]{Browning}. 

\begin{prop}[Browning]\label{prop:sieve}
Let $\mathcal{A}\subset \mathbb{Z}^2,$ and let $F\in\mathbb{Z}[x,X_1,X_2]$ be a polynomial of the form
$$F(x;X_1,X_2) = c_{d}x^{d}+c_{d-1}(X_1,X_2)x^{d-1}+\ldots+c_{0}(X_1,X_2),$$
where $c_d$ is a non-zero integer and $c_i \in \mathbb{Z}[X_1,X_2]$ for every $i\in\{0,\ldots,d-1\}.$ Let $\mathcal{P}$ be a set of primes such that $(c_d,p)=1$ for every $p\in \mathcal{P}$ and $\sqrt{X_1^2+X_2^2} \leq \mathrm{exp}(\#\mathcal{P})$ whenever $(X_1,X_2)\in \mathcal{A}.$ Then, for any integer $\alpha \geq 1,$ we have 
$$\#\{(X_1,X_2)\in \mathcal{A}: F(x;X_1,X_2)=0\text{ for some } x\in \mathbb{Z}\} \ll \frac{1}{\#\mathcal{P}^2}\sum_{p,q\in\mathcal{P}}\bigg|\sum_{i,j\in\{0,1,2\}}c_{i,j}(\alpha)S_{i,j}(p,q)\bigg|,$$
where 
$$S_{i,j}(p,q) = \sum_{\substack{(X_1,X_2)\in \mathcal{A}}}v_{p}(X_1,X_2)^{i}v_{q}(X_1,X_2)^{j},$$
the $v_p(X_1,X_2)$ denote the ``local counts" of solutions 
$$v_p(X_1,X_2)  = \#\{x\,\,(\text{mod}\,\,p): F(x;X_1,X_2) \equiv 0\,\,(\text{mod}\,\,p)\},$$
and the coefficients $c_{i,j}(\alpha)$ are given by 
$$
c_{i,j}(\alpha)=
\begin{cases}
(\alpha-d)^{2}\,\,&\text{if $(i,j)=(0,0)$,} \\
\alpha+(\alpha-1)d-d^2\,\,&\text{if $(i,j)=(1,0)$ or $(0,1)$,} \\
(1+d)^{2}\,\,&\text{if $(i,j)=(1,1)$,} \\
-\alpha-d\,\,&\text{if $(i,j)=(2,0)$ or $(0,2)$,} \\
-1-d\,\,&\text{if $(i,j)=(2,1)$ or $(1,2)$,} \\
1\,\,&\text{if $(i,j)=(2,2)$.} 
\end{cases}
$$
The implied constant is absolute.
\end{prop}
The purpose of the $\alpha$ parameter will become clear later. We will choose it in such a way as to eliminate the ``main term" contribution.

Before we begin, we first make an observation. If $|k| \gg B^{d+1}$ (say) then, by size considerations, we must have $M_{f,g}(B;k)=0.$ Thus, continuing, we may assume that 
\begin{equation}\label{eq:ksize}
0 < |k| \ll B^{d+1}.
\end{equation}
This fact we can restrict to the case when $k$ is polynomially bounded in terms of $B$ will be useful later on.

To apply the determinant method, we began by making a change of variables and proceeded to count points on the simpler 3-dimensional surface $\Gamma_n.$ We will do a similar transformation now for the sieve method. Recall, we wish to count integer points on the affine hypersurface
\begin{equation}
f(x_1,x_2)=(ax_3-bx_4)g(x_3,x_4)+k\subset \mathbb{A}_{\mathbb{Q}}^{4}.
\end{equation}
Unlike the determinant method, the sieve method we will use makes crucial use of the factorisation properties of the above equation. Thus we make a different change of variables. In spite of this change, much of the preliminary work is the same. Let us write
\begin{align} \nonumber
X_1&=x_1, \\ \nonumber
X_2&=x_2, \\ \nonumber
X_3&=ax_3+bx_4, \\ 
h &=ax_3-bx_4.
\end{align}
We view $h$ as fixed and consider counting integer points on the affine surface
\begin{equation}
(2ab)^{d-1} f(X_1,X_2)= (2ab)^{d-1} hg\bigg(\frac{X_3+h}{2a},\frac{X_3-h}{2b}\bigg)+(2ab)^{d-1}k \subset \mathbb{A}_{\mathbb{Q}}^{3}.
\end{equation}
Here we multiply through by suitable powers of $2,a$ and $b$ to ensure that our polynomials have integer coefficients. Let us denote by $K_h(x,y,z,w)$ the projectivisation of this surface in $\mathbb{P}_{\mathbb{Q}}^{3}.$ We can write
\begin{equation}\label{eq:firstsievebound}
M_{f,g}(B;k) \leq  \sum_{0\leq |h|\ll B} \sum_{\substack{1\leq X_1,X_2\leq B \\ 1\leq X_3 \ll B \\ K_{h}(X_1,X_2,X_3,1)=0}}1.
\end{equation}
Exactly as above, we first deal with the degenerate case when $h$ is such that
\begin{equation}
hg\bigg(\frac{X_3+h}{2a},\frac{X_3-h}{2b}\bigg)=0
\end{equation}
identically (as a polynomial in $X_3$). There are $O(1)$ values of $h$ for which this is the case. For these values of $h$ we conclude the contribution to~(\ref{eq:firstsievebound}) is $O_{\epsilon}(B^{3/2+\epsilon}),$ by Lemma~\ref{lemma:splitcomponents}. Our aim is estimate the remaining terms using the polynomial sieve. The sieve method works most effectively when $K_h$ is smooth. Generically this will be true, as the following lemma demonstrates. 

\begin{lemma}
The varieties $K_h\subset \mathbb{P}_{\mathbb{Q}}^{3}$ are smooth for all but at most $O(1)$ values of $h.$
\end{lemma}
\begin{proof}
This is proved in the same way as Lemma~\ref{lemma:singvals} with minor differences.
\end{proof}

It will also be important to have control over various auxiliary curves which arise in the argument. In particular, we will need to have control over the the curves
\begin{equation}\label{eq:Phdef}
(2ab)^{d-1}h\bigg[g\bigg(\frac{x+h}{2a},\frac{x-h}{2b}\bigg)-g\bigg(\frac{y+h}{2a},\frac{y-h}{2b}\bigg)\bigg]\bigg/(x-y)\subset \mathbb{A}_{\mathbb{Q}}^{2}.
\end{equation}
We let $P_h(x,y,w)$ denote the projectivisation of this curve in $\mathbb{P}_{\mathbb{Q}}^{2}.$ For future convenience, we note that 
\begin{align}\nonumber
P_h(x,y,w) &= g_{d-1}(b,a) h (x^{d-2}+x^{d-3}y+\ldots+x y^{d-3}+y^{d-2}) \\ \label{eq:Phdef2}
&+(\text{terms involving $w$})
\end{align}
and 
\begin{equation}\label{eq:Phdef3}
P_h(x,x,1) = (2ab)^{d-1}hg'\bigg(\frac{x+h}{2a},\frac{x-h}{2b}\bigg)
\end{equation}
Here, by the dash notation on the RHS we mean the derivative of the function $g(\frac{x+h}{2a},\frac{x-h}{2b})$ with respect to $x.$ We would also like to restrict to the generic case when $P_h$ is smooth. For this we require the following lemma.

\begin{prop}\label{prop:phsmooth}
The varieties $P_h\subset \mathbb{P}_{\mathbb{Q}}^{2}$ are smooth for all but at most $O(1)$ values of $h.$
\end{prop}
\begin{proof}
We split into two cases, depending on the degree $d$. In this proof $c_1(h),c_2(h)$ and $c_3(h)$ will denote polynomials in $h$ whose coefficients will depend on those of $g,a$ and $b.$ 

The case $d=3$ is simple, as here we have
$$P_h(x,y,w) = g_{d-1}(b,a)h(x+y)+c_1(h)w$$
for some polynomial $c_1(h),$ and it clear that we do not have any singular points whenever $h\neq 0.$ 

Let us now examine the case $d=4$. Here we are going to use the additional assumption (4) we make in the statement of Theorem~\ref{theo:generalmain}. We have
$$P_h(x,y,w) = g_{d-1}(b,a)h(x^2+xy+y^2)+c_2(h)(x+y)w+c_3(h)w^2$$
for some polynomials $c_2(h)$ and $c_3(h).$ From this it is clear that there are no singular points when $w=0$ and $h\neq 0.$ 

Thus we consider possible singular points with $w=1.$ It is easy to check that for the partial derivatives to vanish we must have $x=y,$ and so any singular point is necessarily of the form $[x:x:1]$ where $x$ satisfies, in particular, the system 
$$P_h(x,x,1) = P_h'(x,x,1) = 0.$$
Here the latter quantity denotes the derivative of $P_h(x,x,1)$ with respect to $x$.\footnote{Note $P_h'(t,t,1)= \frac{P_h}{\partial x}(t,t,1)+\frac{P_h}{\partial y}(t,t,1).$} Now, if this is the case, the discriminant $\mathrm{Disc}_x[P_h(x,x,1)]$ must vanish identically. However this is a polynomial in $h$ which generically will be non-zero. Arguing as above, we prove this is non-zero by extracting the leading coefficient and showing this doesn't vanish. From~(\ref{eq:Phdef3}) we obtain
\begin{align*}
\frac{P_h(hx,hx,1)}{h^{d-1}} &= \frac{(2ab)^{d-1}}{h^{d-2}} \sum_{i=0}^{d-1}g_i'\bigg(\frac{hx+h}{2a},\frac{hx-h}{2b}\bigg) \\
&=(2ab)^{d-1}g'_{d-1}\bigg(\frac{x+1}{2a},\frac{x-1}{2b}\bigg)+O(h^{-1})
\end{align*}
%and so we have the pointwise convergence
%$$\frac{P_h(hx,hx,1)}{h^{d-1}} \rightarrow (2ab)^{d-1}g'_{d-1}\bigg(\frac{x+1}{2a},\frac{x-1}{2b}\bigg)$$
%as $h\rightarrow \infty.$ 
Hence, by a similar limiting argument to the proof of Lemma~\ref{lemma:singvals}, we have 
$$\lim_{h\rightarrow \infty} \frac{\mathrm{Disc}_x[P_h(x,x,1)])}{h^{d(d-3)}} = \mathrm{Disc}_x \bigg[(2ab)^{d-1}g'_{d-1}\bigg(\frac{x+1}{2a},\frac{x-1}{2b}\bigg)\bigg].$$
Note that 
$$g'_{d-1}\bigg(\frac{x+1}{2a},\frac{x-1}{2b}\bigg) = \bigg[\frac{1}{2a}\frac{\partial g_{d-1}}{\partial x}+\frac{1}{2b}\frac{\partial g_{d-1}}{\partial y}\bigg]\bigg(\frac{x+1}{2a},\frac{x-1}{2b}\bigg).$$
Hence it is clear that this discriminant doesn't vanish from our assumption that the projective variety
$$ \bigg[\frac{1}{2a}\frac{\partial g_{d-1}}{\partial x}+\frac{1}{2b}\frac{\partial g_{d-1}}{\partial y}\bigg]\bigg(\frac{x}{a},\frac{y}{b}\bigg) \subset \mathbb{P}_{\mathbb{Q}}^{2}$$
is square-free.
\end{proof}

We now deal with the contribution from those values of $h$ for which either $K_h$ or $P_h$ is singular.

\begin{lemma}
Suppose that $h$ is constrained to lie in a set of size $O(1)$ and moreover $h$ is such that 
$$hg\bigg(\frac{X_3+h}{2a},\frac{X_3-h}{2b}\bigg)$$
doesn't vanish identically as a polynomial in $X_3.$ Then the contribution from these values of $h$ to~(\ref{eq:firstsievebound}) is at most $O_{\epsilon}(B^{3/2+\epsilon}).$ 
\end{lemma}
\begin{proof}
This follows in much the same way as the proof of Lemma~\ref{lemma:dealwithsingvals}.
\end{proof}

With this lemma, we can rewrite~(\ref{eq:firstsievebound}) as
\begin{equation}\label{eq:removesing}
M_{f,g}(B;k) \leq \sum_{\substack{0 < |h| \ll B \\ K_h,P_h\text{ smooth}}} \sum_{\substack{1\leq X_1,X_2\leq B \\ 1\leq X_3 \ll B \\ K_{h}(X_1,X_2,X_3,1)=0}}1+O_{\epsilon}(B^{3/2+\epsilon}).
\end{equation}
For each fixed $h$ we will apply the polynomial sieve to count the inner sum, by detecting solubility of the equation in the $X_3$ variable. In the notation of Proposition~\ref{prop:sieve}, we will take 
\begin{equation}
\mathcal{A}:= \{(X_1,X_2)\in ([1,B]\cap \mathbb{Z})^2: (2ab)^{d-1}f(X_1,X_2)\equiv (2ab)^{d-1}k\,\,(\text{mod}\,\,|h|)\}
\end{equation}
and
\begin{equation}
F(x;X_1,X_2) := K_h(X_1,X_2,x,1).
\end{equation}
After unravelling the definition of $K_h$ we find that
\begin{equation}
F(x;X_1,X_2) = hg_{d-1}(b,a)x^{d-1}+(\text{lower order terms in $x$}).
\end{equation}
We are assuming that $hg_{d-1}(b,a)\neq 0.$\footnote{Recall that our assumption that the projective variety $\{(ax-by)g_{d-1}(x,y)\}\subset \mathbb{P}_{\mathbb{Q}}^{1}$ is square-free implies, in particular, that $g_{d-1}(b,a)\neq 0.$}  In particular, $F$ is of the correct form to apply Proposition~\ref{prop:sieve}. We define
\begin{equation}
\mathcal{P} = \{p\leq Q: p\nmid 6abg_{d-1}(b,a)\mathrm{cont}(f_d)\mathrm{Disc}[f_d]\mathrm{Disc}[K_h]\mathrm{Disc}[P_h] h\}
\end{equation}
for some large value of $Q$. By $\mathrm{cont}(f_d)$ we mean the content of the polynomial $f_d$ (i.e. the gcd of all the coefficients). Here, whenever $F$ is a homogenous polynomial we denote by $\mathrm{Disc}[F]$ its discriminant. This choice is important, as it will ensure our varieties remain smooth when viewed over $\overline{\mathbb{F}}_p$ (for any $p\in \mathcal{P}$). 

We need to be careful here, as this set clearly depends on $h$ and will also depend (in some complicated manner) on $k.$ Now, it is a standard fact that the discriminant of a homogenous polynomial of degree $N$ is a homogenous polynomial of degree $(2N-2)$ in the coefficients. This, together with the size bounds $|h|\ll B$ and $|k| \ll B^{d+1}$ (recall~(\ref{eq:ksize})) imply that 
$$\mathrm{Disc}[K_h]\mathrm{Disc}[P_h] h\ll B^{5d^2}$$
uniformly in $h$ and $k$ (say). We will eventually take $Q=B^{\delta}$ for some small $\delta>0.$ This means, by the prime number theorem, we will have the asymptotic
\begin{equation}
\#\mathcal{P}\sim Q/\log{Q}
\end{equation}
uniformly in $h$ and $k,$ whenever $B\gg 1$.

Thus, applying Proposition~\ref{prop:sieve} with our choices above to each term in the sum~(\ref{eq:removesing}) yields
\begin{equation}\label{eq:applysieve}
M_{f,g}(B;k) \ll_{\epsilon}  \sum_{\substack{0 < |h| \ll B \\ K_h,P_h\text{ smooth}}}  \frac{1}{(\#\mathcal{P})^{2}}\sum_{p,q\in\mathcal{P}}\bigg|\sum_{i,j\in\{0,1,2\}}c_{i,j}(\alpha)S_{i,j}(p,q)\bigg|+B^{3/2+\epsilon},
\end{equation}
where 
\begin{align}
S_{i,j}(p,q) &= \sum_{\substack{(X_1,X_2)\in \mathcal{A}}}v_{p}(X_1,X_2)^{i}v_{q}(X_1,X_2)^{j}, \\
v_p(X_1,X_2)  &= \#\{x\,\,(\text{mod}\,\,p): K_h(X_1,X_2,x,1) \equiv 0\,\,(\text{mod}\,\,p)\},
\end{align}
and the constants $c_{i,j}(\alpha)$ are defined as in the statement of Proposition~\ref{prop:sieve}. 

We evaluate the sums $S_{i,j}(p,q)$ following the method of Browning~\cite{Browning}, by first restricting to congruence classes modulo $pq|h|$ and then completing exponential sums. We have
\begin{align} \nonumber
S_{i,j}(p,q) &= \sum_{r,s\,\,(\text{mod}\,\,pq|h|)} \sum_{\substack{(X_1,X_2)\in\mathcal{A}  \\ X_1\equiv r\,\,(\text{mod}\,\,pq|h|)\\ X_2\equiv s\,\,(\text{mod}\,\,pq|h|)  }} v_p(X_1,X_2)^{i}v_q(X_1,X_2)^{j} \\
&=  \sum_{\substack{r,s\,\,(\text{mod}\,\,pq|h|) \\ (2ab)^{d-1}f(r,s)\equiv (2ab)^{d-1}k\,\,(\text{mod}\,\,|h|)}} v_p(r,s)^{i}v_q(r,s)^{j} \sum_{\substack{1\leq X_1,X_2\leq B  \\ X_1\equiv r\,\,(\text{mod}\,\,pq|h|)\\ X_2\equiv s\,\,(\text{mod}\,\,pq|h|)}} 1.
\end{align}
We can detect the congruence condition in the inner sum using additive characters, as follows:
\begin{align} \nonumber
\sum_{\substack{1\leq X_1 \leq B \\ X_1\equiv r\,\,(\text{mod}\,\,pq|h|)}}1 &= \frac{1}{pq|h|}\sum_{\substack{1\leq X_1 \leq B }}\sum_{m\,\,(\text{mod}\,\,pq|h|)} e\bigg(-\frac{m(X_1-r)}{pq|h|}\bigg) \\ \nonumber
&=  \frac{1}{pq|h|}\sum_{-pq|h|/2 < m \leq pq|h|/2} e\bigg(\frac{mr}{pq|h|}\bigg)\sum_{\substack{1\leq X_1 \leq B }}e\bigg(-\frac{mX_1}{pq|h|}\bigg) \\
&:=\frac{1}{pq|h|}\sum_{-pq|h|/2 < m \leq pq|h|/2} \Gamma(B,m)e\bigg(\frac{mr}{pq|h|}\bigg),
\end{align}
where we have defined
\begin{equation}
\Gamma(B,m) := \sum_{\substack{1\leq l\leq B }}e\bigg(\frac{-ml}{pq|h|}\bigg).
\end{equation}
We note the well-known bound here
\begin{equation}
\Gamma(B,m) \ll \min\bigg\{B,\frac{pq|h|}{|m|}\bigg\},
\end{equation}
which will be used later. A similar identity holds for the sum over $X_2.$ Putting these facts together, and then swapping sums, we obtain the expression
\begin{align}
S_{i,j}(p,q) &= \frac{1}{(pqh)^{2}} \sum_{-pq|h|/2<m,n\leq pq|h|/2} \Gamma(B,m)\Gamma(B,n)\Psi_{i,j}(m,n),
\end{align}
where
\begin{equation}
\Psi_{i,j}(m,n) := \sum_{\substack{r,s\,\,(\text{mod}\,\,pq|h|) \\ (2ab)^{d-1}f(r,s)\equiv (2ab)^{d-1}k\,\,(\text{mod}\,\,|h|)}} v_p(r,s)^{i}v_q(r,s)^{j}  e\bigg(\frac{mr+ns}{pq|h|}\bigg).
\end{equation}
By our choice of $\mathcal{P}$ we have $(pq,h)=1.$ This allows us to deduce the following multiplicativity property for the exponential sums $\Psi_{i,j}.$
\begin{lemma}\label{lemma:multi}
The following factorisations hold.
\begin{enumerate}
	\item Suppose $p\neq q$ and let $p',q',\overline{pq},\overline{h}\in\mathbb{Z}$ be defined by $pq\overline{pq}+|h|\overline{h}=1$ and $pp'+qq'=1.$ Then 
	$$\Psi_{i,j}(m,n)=\Sigma_i(p;\overline{h}q'm,\overline{h}q'n)\Sigma_j(q;\overline{h}p'm,\overline{h}p'n)\Phi(|h|;\overline{pq}m,\overline{pq}n).$$
	\item Suppose $p= q$ and let $\overline{p},\overline{h}\in\mathbb{Z}$ be defined by $p\overline{p}+|h|\overline{h}=1.$ Then 
	$$
	\Psi_{i,j}(m,n)=
	\begin{cases}
	p^2\Sigma_{i+j}(p;\overline{h}m',\overline{h}n')\Phi(|h|;\overline{p}m',\overline{p}n')\,\,&\text{if $(m,n)=p(m',n')$,} \\
	0\,\,&\text{otherwise.}
	\end{cases}
	$$
\end{enumerate}
Here 
\begin{equation}\label{eq:defsigma}
\Sigma_t(p;M,N) = \sum_{x,y\in \mathbb{F}_p}v_p(x,y)^t e\bigg(\frac{Mx+Ny}{p}\bigg)
\end{equation}
and
\begin{equation}\label{eq:defphi}
\Phi(h;M,N) = \sum_{\substack{x,y\,\,(\text{mod}\,\,h) \\ (2ab)^{d-1}f(x,y)\equiv (2ab)^{d-1}k\,\,(\text{mod}\,\,h)}}e\bigg(\frac{Mx+Ny}{h}\bigg).
\end{equation}
\end{lemma}
\begin{proof}
This is~\cite[Lemma 3.4]{Browning} with slight changes to notation.
\end{proof}
Recall $i,j\in\{0,1,2\}.$ Thus, to examine the exponential sums $\Psi_{i,j}$ we may restrict our analysis to the exponential sums $\Sigma_t$ for $0\leq t\leq 4$ and $\Phi.$ 

For the former, it is important that we have restricted to the case where our varieties are smooth; the desired bounds will then follow relatively straightforwardly from the work of Weil and Deligne. We will make use of Hooley's method of moments, which allows us to estimate an exponential sum over an algebraic variety by counting points on the variety over finite fields. The following result originates in~\cite{Hooley4} and appears in the form stated here as~\cite[Lemma 3.5]{Browning}.

\begin{lemma}[Hooley's method of moments]\label{lemma:HMOM}
Let $F,G_1,\ldots,G_k$ be polynomials over $\mathbb{Z}$ of degree at most $d$, and let 
$$S =\sum_{\substack{{\bf x}\in \mathbb{F}_p^n \\ G_1({\bf x})=\ldots = G_k({\bf x}) = 0}}e\bigg(\frac{F({\bf x})}{p}\bigg)$$
for any prime $p.$ For each $j\geq 1$ and $\tau \in \mathbb{F}_{p^{j}}$ we define the sets
$$N_j(\tau) = \#\{{\bf x}\in \mathbb{F}_{p^j}^{n}:G_1({\bf x})=\ldots = G_k({\bf x}) = 0\text{ and } F({\bf x}) = \tau \}.$$
Suppose there exists $N_j\in\mathbb{R}$ such that 
$$\sum_{\tau \in \mathbb{F}_{p^j}}|N_j(\tau)-N_j|^{2} \ll_{d,k,n} p^{\kappa j}$$
where $\kappa\in \mathbb{Z}$ is independent of $j.$ Then $S \ll_{d,k,n} p^{\kappa/2}.$
\end{lemma}

Once we have reduced to this point-counting problem over finite fields, we may employ the work of Weil~\cite{Weil} to count points on curves, and the work of Deligne~\cite{Deligne} to count points on higher-dimensional varieties. The following two results will be sufficient for our purposes.

\begin{lemma}[Deligne]\label{lemma:deligne} Let $W\subset \mathbb{P}_{\mathbb{F}_q}^n$ be a non-singular complete intersection of dimension 2 and degree $d$. Then 
$$\#\{{\bf x}\in \mathbb{F}_{q}^n: [{\bf x}] \in W\} = q^3+O_{d,n}(q^2).$$
\end{lemma}
\begin{lemma}[Weil]\label{lemma:weil} Let $V\subset \mathbb{A}_{\mathbb{F}_q}^n$ be an absolutely irreducible curve of degree $d$. Then 
$$\#\{{\bf x}\in \mathbb{F}_{q}^n: {\bf x} \in V\}  =  q+O_{d,n}(q^{1/2}).$$
\end{lemma}

Recall that a projective variety $W\subset \mathbb{P}_{\mathbb{F}_p}^{n}$ is a {\it complete intersection} if it is generated by exactly $\mathrm{codim}(W)$ elements.

We will estimate the exponential sums $\Phi$ by obtaining square-root cancellation when $l=1$ (in the generic case), and using elementary arguments for higher powers. This is the part of the argument where we will use the fact the curve $f(x,y)=k$ doesn't contain any lines definable over $\mathbb{Q}.$

In both cases, care must be taken to ensure that our results have no dependence on the constant term $k$.  We recall our convention, adopted in Section~\ref{section:notation}, which says all implied constants are allowed to depend on $f,g,a$ and $b$ without specifying so, and this includes dependencies on the coefficients of $f$ and $g$ as well as on the degree $d$. 

Finally, the following two sections will require us to perform arithmetic in $\overline{\mathbb{F}}_p.$ To this end, we adopt the convention that any rational number $a/b$ whose denominator is coprime to $p$ may be viewed as an element of $\overline{\mathbb{F}}_p,$ namely $\overline{a}\overline{b}^{-1},$ where $\overline{x}$ denotes the reduction map modulo $p$ and $\overline{x}\overline{x}^{-1}\equiv 1\,\,(\text{mod}\,\,p)$. In particular, the rational number $a/b$ vanishes when viewed as an element of $\overline{\mathbb{F}}_p$ if and only if $p|a.$

%%%%%%%%%%%%%%
%%%%%%%%%%%%%%
%%%%%%%%%%%%%%
%%%%%%%%%%%%%%
%%%%%%%%%%%%%%
\section{Estimation of exponential sums (I)}

Recall the definition of $\Sigma_t$ given by~(\ref{eq:defsigma}). In this section we are going to estimate the exponential sums
\begin{equation}
\Sigma_t(p;M,N) = \sum_{x,y\in \mathbb{F}_p}v_p(x,y)^t e\bigg(\frac{Mx+Ny}{p}\bigg)
\end{equation}
for $0\leq t\leq 4$ and integers $M,N$. We argue in much the same way as Browning~\cite{Browning} did for the analogous sums he encountered when investigating the symmetric additive energy of quartic polynomials, with a few changes. Our main result is the following (cf.~\cite[Lemma 4.3]{Browning}).

\begin{prop}\label{lemma:browningexp}
Let $p\in \mathcal{P}$. For $0\leq t\leq 4$ we have 
$$\Sigma_t(p;M,N) \ll p(p,M,N)$$
and for $0\leq t\leq 2$ we have 
$$\Sigma_t(p;0,0) = \max\{1,t\} p^2+O(p).$$
\end{prop}
%\begin{proof}
%This is proved in much the same way as~\cite[Lemma 4.3]{Browning}, because we are in the regime where $K_h$ and $P_h$ are smooth. There is a minor difference in our analysis for the sums $\Sigma_3$ and $\Sigma_4.$ Here we note that, when $d=3,$ we have
%$$P_h(x,y,1) = P_h(x,z,1) \iff y=z,$$
%and this fact simplifies the analysis. In the case $d=4,$ we have
%$$P_h(x,y,1) = P_h(x,z,1) \iff y=z\text{ or } x+y+z=g_{d-1}(b,a)^{-1}c(h)$$
%where $c(h)$ is a linear polynomial in $h$ with integer coefficients. In Browning's work $c(h)=0$ but for us it may be that $c(h)\neq 0.$ The difference is minimal and doesn't affect the final result (cf.~ the argument just preceding~\cite[equation (4.9)]{Browning}). By keeping track of any dependencies, it is clear that we may allow any implied constants to depend on $f_{d},g_{d-1},a$ and $b.$
%\end{proof}
We will prove this result in stages. The case $t=0$ follows immediately from orthogonality of additive characters. Thus we turn our attention to the case $t=1,$ where
\begin{equation}\label{eq:sigma1}
\Sigma_1(p;M,N)  = \sum_{\substack{x,y,z\,\,(\text{mod}\,\,p) \\ K_h(x,y,z,1)=0}}e\bigg(\frac{Mx+Ny}{p}\bigg).
\end{equation}
We begin by proving the following lemma.
\begin{lemma}\label{lemma:simpledisc}
For any integers $N$ and $M$ we have 
\begin{equation}
\mathrm{Disc}_x [f_d(Nx,1-Mx)] = N^D \mathrm{Disc}_x [f_d(x,1)]
\end{equation}
for some integer $D$ depending on $f_d,M$ and $N$.
\end{lemma}
\begin{proof}
%There is {\it surely} a way to prove this by hand, but I can't see it at the moment. In any case, we only need this result for $d\in\{3,4\}$ and here we can check it via Mathematica.
%\\ \\
The result is true if $N=0$ as then both sides equal zero. Thus we may suppose that $N\neq 0.$ We have
$$\mathrm{Disc}_x [f_d(Nx,1-Mx)]  = N^{D}\mathrm{Disc}_x [f_d(x,1-Mx/N)]$$
where $D:=l(l-1)$ and $l$ is the degree of $f_d(x,1-Mx/N)$ as a polynomial in $x$. To prove the lemma, it suffices to show that the discriminant $\mathrm{Disc}_x (f_d(x,1-tx)),$ which by definition is a polynomial in $t$, is in fact constant in $t$. To do this, we require two further properties of discriminant polynomials which we state here. Namely, for any polynomial $F$ we have 
$$\mathrm{Disc}_x(xF(x)) =  F(0)^2\mathrm{Disc}_x(F(x))$$
and, if in addition $F$ has degree $d$ and $F(0)\neq 0,$ we have 
$$\mathrm{Disc}_x(F(x)) = \mathrm{Disc}_x(x^d F(1/x)).$$
We split into two cases.
\begin{enumerate}
	\item Suppose $f_d(0,1)\neq 0.$ In this case, provided $t$ is such that $f_d(1,-t)\neq 0,$ we have 
\begin{align*}
\mathrm{Disc}_x (f_d(x,1-tx)) &= \mathrm{Disc}_x (x^d f_d(1,1/x-t)) \\
&=  \mathrm{Disc}_x (f_d(1,x-t)) \\
&=\mathrm{Disc}_x ( f_d(1,x)).
\end{align*}
As both sides are polynomials in $t$, we conclude this identity in fact holds for all $t.$ The result follows.
	\item Suppose $f_d(0,1)=0.$ Then we must have $(\partial f_d/\partial x)(0,1)\neq 0$ as otherwise $f_d(x,y)$ would be divisible by $x^2.$ Again supposing $t$ is such that $f_d(1,-t)\neq 0,$ we have 
\begin{align*}
\mathrm{Disc}_x (f_d(x,1-tx)) &= \mathrm{Disc}_x (x^{d} f_d(1,1/x-t)) \\
&= \bigg[\lim_{x\rightarrow 0} \frac{f_d(x,1-xt)}{x}\bigg]^2 \mathrm{Disc}_x (x^{d-1}f_d(1,1/x-t)) \\
&= \bigg[\frac{\partial f_d}{\partial x}(0,1)\bigg]^2\mathrm{Disc}_x (f_d(1,x))
\end{align*}
The result follows, as above.
\end{enumerate}
\end{proof}

We isolate the $t=1$ case of Proposition~\ref{lemma:browningexp} in the following lemma.
\begin{lemma}\label{lemma:t1case}
Let $p\in \mathcal{P}$. Then
$$
\Sigma_1(p;M,N)
= 
\begin{cases}
p^2+O(p)\,\,&\text{if $p|(M,N),$} \\
O(p)\,\,&\text{otherwise.}
\end{cases}
$$
\end{lemma}
\begin{proof}
First we consider the case $p|(M,N).$ We may write
\begin{equation*}
\Sigma_1(p;0,0)  = \frac{1}{p-1}\#\{x,y,z,w\in \mathbb{F}_p: K_h(x,y,z,w)=0\text{ and } w\neq0\}.
\end{equation*}
Our set-up ensures that $K_h(x,y,z,w)$ is a non-singular projective surface over $\mathbb{F}_p.$ The hypotheses of Lemma~\ref{lemma:deligne} are satisfied, and so we obtain
\begin{equation*}
\#\{x,y,z,w\in \mathbb{F}_p: K_h(x,y,z,w)=0\} = p^3+O(p^2).
\end{equation*}
We have $K_h(x,y,z,0) = (2ab)^{d-1}f_d(x,y).$ Assuming $p\in \mathcal{P},$ it follows that
\begin{equation*}
\#\{x,y,z\in \mathbb{F}_p: K_h(x,y,z,0)=0\} \ll p^2.
\end{equation*}
Putting these facts together yields
\begin{equation*}
\Sigma_1(p;0,0)  = p^2+O(p).
\end{equation*}
Now let us suppose $p\nmid (M,N).$ We will use Lemma~\ref{lemma:HMOM} to show that (generically) we have square-root cancellation. To this end, fix an integer $j\geq 1$ and $\tau \in \mathbb{F}_{p^{j}},$ and define
\begin{equation*}
N_{j}(\tau)  := \#\{x,y,z \in \mathbb{F}_{p^{j}}: K_h(x,y,z,1) = 0 \text{ and } Mx+Ny= \tau\}.
\end{equation*}
We may suppose WLOG that $p\nmid N,$ with the case $p\nmid M$ being treated similarly. 
%Because  $K_h(x,y,z,w)$ is a non-singular, the polynomial $K_h(x,y,z,1)$ is absolutely irreducible over $\mathbb{F}_p.$  
We may rewrite the above as
\begin{equation*}
N_{j}(\tau) = \#\{x,z \in \mathbb{F}_{p^{j}}: K_h(x,(\tau-Mx)N^{-1},z,1) = 0\}.
\end{equation*}
We claim the following.
\\ \\
{\it Claim.} $K_h(x,(\tau-Mx)N^{-1},z,1) =0$ defines an absolutely irreducible curve for all but at most $O(1)$ values of $\tau$.
\begin{proof}[Proof (of claim)]
To examine whether the curve $K_h(x,(\tau-Mx)N^{-1},z,1) = 0$ is absolutely irreducible, we investigate the smoothness properties of its projectivisation over $\overline{\mathbb{F}}_p$. By Taylor expansion (see e.g.~(\ref{eq:taylorexpansion})), and using the fact $p\nmid N,$ we may write 
\begin{align*}
f(x,(\tau-Mx)N^{-1}) &=\frac{f_d(N,-M)}{N^d}x^d \\
&+ \bigg[\frac{f_{d-1}(N,-M) }{N^{d-1}}- \tau\frac{(\partial f_d/\partial y)(N,-M)}{N^d} \bigg]x^{d-1}+\ldots.
\end{align*}
We also have 
\begin{equation*}
(2ab)^{d-1}hg\bigg(\frac{z+h}{2a},\frac{z-h}{2b}\bigg) = g_{d-1}(b,a)hz^{d-1}+\ldots.
\end{equation*}
There are two cases to consider. 
\begin{enumerate}
	\item Suppose $f_d(N,-M)=0$ in $\overline{\mathbb{F}}_{p}.$ In this case, since $f_d$ is smooth and $N\neq 0$ we have $(\partial f_d/\partial y)(N,-M)\neq 0.$ We exclude the value $\tau = Nf_{d-1}(N,-M)/(\partial f_d/\partial y)(N,-M),$ as we may, and then we see the projectivisation of this curve can be written as
	\begin{align*}
	\bigg[\frac{f_{d-1}(N,-M) }{N^{d-1}}- \tau\frac{(\partial f_d/\partial y)(N,-M)}{N^d} \bigg]x^{d-1} &=g_{d-1}(b,a)hz^{d-1} + (\text{terms involving $w$}).
	\end{align*} 
	It is then clear, by considering vanishing of partial derivatives, that there cannot be any singular points with $w=0.$ This leaves us to investigate singular points of the form $[r:s:1].$ By considering the $z$-derivative, we see that $s$ is constrained to at most $O(1)$ values. For each such value, there will be a solution in $r$ if and only if the discriminant
	\begin{equation}\label{eq:discforKh}
	\mathrm{Disc}_x\bigg(f(x,(\tau-Mx)N^{-1}) - (2ab)^{d-1}hg\bigg(\frac{s+h}{2a},\frac{s-h}{2b}\bigg) -(2ab)^{d-1}k\bigg)
	\end{equation}
	vanishes. This will be a polynomial in $\tau$ which generically is non-zero, and so will only vanish for $O(1)$ values of $\tau.$ As previously, we will prove this by extracting the leading coefficient and showing this is non-zero. We have 
\begin{align*}
\frac{f(\tau x, (\tau-M\tau x)N^{-1}}{\tau^d} = f_d(x,(1-Mx)N^{-1})+O(\tau^{-1}).
\end{align*}
By similar arguments to the proof of Lemma~\ref{lemma:singvals}, it follows that the leading coefficient of the dscriminant~(\ref{eq:discforKh}), as a polynomial in $\tau,$ is 
\begin{equation*}
\mathrm{Disc}_x( f_d(x,(1-Mx)N^{-1}).
\end{equation*}
By Lemma~\ref{lemma:simpledisc} and our assumptions on $p$, this doesn't vanish.
	\item Suppose $f_d(N,-M)\neq 0$ in $\overline{\mathbb{F}}_{p}.$ Then the projectivised curve can be written as 
	\begin{align*}
	\frac{f_d(N,-M)}{N^d} x^{d} &=g_{d-1}(b,a)hz^{d-1}w + (\text{terms involving $w$}).
	\end{align*} 
	In exactly the same way as above, we conclude that there are no singular points with $w=0$ and for any singular point of the form $[r:s:1],$ $s$ is constrained to at most $O(1)$ values and for each such value there exists an $r$ if and only if the discriminant defined by~(\ref{eq:discforKh}) vanishes. This can only occur for at most $O(1)$ values of $\tau.$
\end{enumerate}
This finishes the proof of the claim.
\end{proof}
Now, for the $O(1)$ values of $\tau$ for which the curve $K_h(x,(\tau-Mx)N^{-1},z,1) =0$ is not absolutely irreducible, we will apply the trivial bound $N_{j}(\tau) \ll p^{j}.$ In the complementary case, by Lemma~\ref{lemma:weil}, we obtain 
\begin{equation*}
N_{j}(\tau) = p^{j} + O(p^{j/2})
\end{equation*}
We may therefore take $N_{j} = p^{j}$ and $\kappa=2$ in the statement of Lemma~\ref{lemma:HMOM}, and the result follows.
\end{proof}

For the case $t=2,$ we must examine the exponential sum
\begin{align}
\Sigma_2(p;M,N)&=\sum_{\substack{x,y,z_1,z_2\,\,(\text{mod}\,\,p) \\ K_h(x,y,z_1,1)=0 \\ K_h(x,y,z_2,1)=0}}e\bigg(\frac{Mx+Ny}{p}\bigg).
\end{align}
Recalling the definition of $P_h$ in~(\ref{eq:Phdef}), we see that 
\begin{equation}
K_h(x,y,z_1,1)= K_h(x,y,z_2,1) \iff (z_1-z_2)P_h(z_1,z_2,1)=0.
\end{equation}
Thus we may equivalently write 
\begin{align}
\Sigma_2(p;M,N) &=\sum_{\substack{x,y,z_1,z_2\,\,(\text{mod}\,\,p) \\ K_h(x,y,z_1,1)=0 \\ (z_1-z_2)P_h(z_1,z_2,1)=0}}e\bigg(\frac{Mx+Ny}{p}\bigg).
\end{align}
We isolate the $t=2$ case of Proposition~\ref{lemma:browningexp} in the following lemma.
\begin{lemma}
Let $p\in \mathcal{P}.$ Then 
$$
\Sigma_2(p;M,N)
= 
\begin{cases}
2p^2+O(p)\,\,&\text{if $p|(M,N),$} \\
O(p)\,\,&\text{otherwise.}
\end{cases}
$$
\end{lemma}
\begin{proof}
Separating out the contribution from $z_1=z_2,$ and recalling the definition of $\Sigma_1$ (see~(\ref{eq:sigma1}) above), we may write 
\begin{align}\label{eq:sigma12}
\Sigma_2(p;M,N) &=\Sigma_1(p;M,N)+\sum_{\substack{x,y,z_1,z_2\,\,(\text{mod}\,\,p) \\ z_1\neq z_2 \\ K_h(x,y,z_1,1)=0 \\ P_h(z_1,z_2,1)=0}}e\bigg(\frac{Mx+Ny}{p}\bigg).
\end{align}
In this last sum we may add back in the terms with $z_1=z_2,$ as these contribute
\begin{align} \label{eq:sigma12next}
\leq \sum_{\substack{z\,\,(\text{mod}\,\,p) \\ P_h(z,z,1)=0}}\sum_{\substack{x,y\,\,(\text{mod}\,\,p) \\ K_h(x,y,z,1)=0}}1 \ll p \sum_{\substack{z\,\,(\text{mod}\,\,p) \\ P_h(z,z,1)=0}}1 \ll p.
\end{align}
Overall, we obtain 
\begin{align*}
\Sigma_2 &= \Sigma_1+\sum_{\substack{x,y,z_1,z_2\,\,(\text{mod}\,\,p) \\ K_h(x,y,z_1,1)=0 \\ P_h(z_1,z_2,1)=0}}e\bigg(\frac{Mx+Ny}{p}\bigg) +O(p).
\end{align*} 
Let us call this inner sum $T(p;M,N).$ The following claim is important.
\\ \\
{\it Claim.} For any $p\in \mathcal{P}$ the variety 
\begin{equation}
\{K_h(x,y,z_1,w) = P_h(z_1,z_2,w) = 0 \} \subset \mathbb{P}_{\overline{\mathbb{F}}_p}^{4}
\end{equation}
is smooth.
\begin{proof}[Proof (of claim)]
Any singular point by definition must solve the system
\begin{align*}
K_h = P_h &= 0,\,\,\lambda \nabla K_h = \mu \nabla P_h
\end{align*}
for some $(\lambda,\mu)\neq (0,0).$ By the work above we know that both $K_h$ and $P_h$ are non-singular over $\mathbb{Q}.$ It follows that we must have $\lambda \mu \neq 0.$ Now, the gradient identity implies the following equations must be solved:
\begin{align*}
\frac{\partial{K_h}}{\partial{x}} &= 0, \\
\frac{\partial{K_h}}{\partial{y}} &= 0, \\
\lambda \frac{\partial{K_h}}{\partial{z_1}} &= \mu\frac{\partial{P_h}}{\partial{z_1}},   \\
\frac{\partial{P_h}}{\partial{z_2}}  &= 0, \\
\lambda \frac{\partial{K_h}}{\partial{w}} &= \mu\frac{\partial{P_h}}{\partial{w}}.
\end{align*}
We first rule out the possibility of any singular points with $w=0$. If $[r:s:t:u:0]$ is a singular point, then the first two equations imply that 
$$ \frac{\partial f_d}{\partial x}(r,s) = \frac{\partial f_d}{\partial y}(r,s) =0.$$
By Euler's identity we conclude that $f_d(r,s)=0.$ As we are assuming $f_d$ is smooth over $\mathbb{F}_p,$ we must therefore have $r=s=0.$ Now $\partial{K_h}/\partial{z_1}$ vanishes when $w=0.$ Recalling the definition of $P_h$ (see~(\ref{eq:Phdef3})), we have 
$$P_h(x,y,w) = g_{d-1}(b,a)h (x^{d-2}+x^{d-3}y+\ldots+x y^{d-3} +y^{d-2})+\ldots.$$
If $d=3$ then the third equation (say) cannot be satisfied, and so there are no singular points in this case. Thus we may suppose that $d=4.$ In this case the third and fourth equation imply that 
\begin{align*}
2t+u &= 0 \\
t+2u &= 0,
\end{align*}
whence we must have $t=u=0.$ This is a contradiction, and so there cannot be any singular points with $w=0.$ We now consider singular points of the form $[r:s:t:u:1].$ The fourth equation implies that $(\partial{P_h}/\partial{z_2})(t,u)=0.$ Likewise, on replacing $K(x,y,z_1,w)$ by $K(x,y,z_2,w)$ and using the symmetry of $P_h(z_1,z_2,w)$ in the first two variables, we see that $(\partial{P_h}/\partial{z_1})(r,s)=0$ also. Since we are assuming $P_h(r,s,1)=0,$ by Euler's identity we see that $(\partial{P_h}/\partial{w})(r,s)=0$. But now we have produced a singular point on the curve $P_h,$ a contradiction.
%\footnote{
%For this last bit, we argue as follows: we are investigating the smoothness of the variety $V_1 = \{K_h(x,y,z_1,w) = P_h(z_1,z_2,w) = 0 \} \subset \mathbb{P}_{\mathbb{Q}}^{4}.$ The $z_2$-derivative for this says that any point $[x^*:y^*:z_1^*:z_2^*:w^*]$ must satisfy 
%$$\frac{\partial{P_h}}{\partial{z_2}}(z_1^*,z_2^*,w^*)=0.$$
%Now, by swapping the variables $z_1,z_2,$ we may trivially rewrite our variety as $V_1 = \{K_h(x,y,z_2,w) = P_h(z_2,z_1,w) = 0 %\} \subset \mathbb{P}_{\mathbb{Q}}^{4}.$ Using the symmetry of $P_h$ in the first two arguments, this is $V_1 = \{K_h(x,y,z_2,w) = P_h(z_1,z_2,w) = 0 \} \subset \mathbb{P}_{\mathbb{Q}}^{4}.$ Now the $z_1$-derivative says that any point must also satisfy
%$$\frac{\partial{P_h}}{\partial{z_1}}(z_1^*,z_2^*,w^*)=0.$$
%}
\end{proof}

With this claim proven, we move on to estimating the exponential sum $T(p;M,N)$ defined above. We begin by considering the case $p|(M,N).$ By definition,
\begin{equation*}
T(p;0,0) = \#\{x,y,z_1,z_2 \in \mathbb{F}_p:  K_h(x,y,z_1,1) = P_h(z_1,z_2,1)=0 \},
\end{equation*}
which we write as
\begin{equation*}
T(p;0,0) =\frac{1}{p-1} \#\{x,y,z_1,z_2,w \in \mathbb{F}_p:  K_h(x,y,z_1,w)=P_h(z_1,z_2,w)=0 \text{ and } w\neq 0\}.
\end{equation*}
This set defines a non-singular, complete intersection, projective variety of dimension 2. Thus, we may apply Lemma~\ref{lemma:deligne} to conclude that
\begin{equation*}
\#\{x,y,z_1,z_2,w \in \mathbb{F}_p:  K_h(x,y,z_1,w)=P_h(z_1,z_2,w)=0 \}= p^3+O(p^2).
\end{equation*}
The contribution from points with $w=0$ is
\begin{equation*}
\sum_{\substack{z_1,z_2\in \mathbb{F}_p \\ P_h(z_1,z_2,0)=0}}\sum_{\substack{x,y\in\mathbb{F}_p\\ K_h(x,y,z_1,0)=0}}1 \ll p\sum_{\substack{z_1,z_2\in \mathbb{F}_p \\ P_h(z_1,z_2,0)=0}}1 \ll p^2.
\end{equation*}
It follows that 
\begin{equation*}
T(p;0,0)  = p^2+O(p).
\end{equation*}
Now let us suppose $p\nmid (M,N).$ We will use Lemma~\ref{lemma:HMOM} to show that (generically) we have square-root cancellation. Fix $j\geq 1$ and $\tau \in \mathbb{F}_{p^{j}}.$ Then we must count 
\begin{equation*}
N_{j}(\tau):=\#\{x,y,z_1,z_2 \in \mathbb{F}_{p^{j}}:  K_h(x,y,z_1,1)=P_h(z_1,z_2,1)=0 \text{ and } Mx+Ny=\tau \}.
\end{equation*}
We may suppose WLOG that $p\nmid N.$ Then we may write this as 
\begin{equation*}
N_{j}(\tau)=\#\{x,z_1,z_2 \in \mathbb{F}_{p^{j}}:  K_h(x,(\tau-Mx)N^{-1},z_1,1)=P_h(z_1,z_2,1)=0\}.
\end{equation*}
We claim the following.
\\ \\
{\it Claim.} $K_h(x,(\tau-Mx)N^{-1},z_1,1)=P_h(z_1,z_2,1)=0$ defines an absolutely irreducible curve for all but at most $O(1)$ values of $\tau$.
\begin{proof}[Proof (of claim)]
This proceeds much the same way as the analogous claim contained in Lemma~\ref{lemma:t1case}. We examine the absolute irreducibility of this curve by investigating the smoothness properties of its projectivisation over $\overline{\mathbb{F}}_p$. 
%We recall that 
%\begin{align*}
%P_h(z_1,z_2,1) &= h g_{d-1}(b,a)(z_1^{d-2}+z_1^{d-3}z_2+\ldots+z_1z_2^{d-3}+z_2^{d-2})+\ldots.
%\end{align*}
There are two cases to consider.
\begin{enumerate}
	\item If $f_d(N,-M)=0$ in $\overline{\mathbb{F}}_p,$ then we exclude the value $\tau = Nf_{d-1}(N,-M)/(\partial f_d/\partial y)(N,-M),$ as we may, and consider the equations
	\begin{align*}
	\bigg[\frac{f_{d-1}(N,-M) }{N^{d-1}}- \tau\frac{(\partial f_d/\partial y)(N,-M)}{N^d} \bigg]x^{d-1} &=g_{d-1}(b,a)hz_1^{d-1} + (\text{terms involving $w$})
	\end{align*}
	and
	\begin{align*}
	P_h(z_1,z_2,w)= 0.
	\end{align*}
	Call the first equation $\psi(x,z_1,w)=0.$ We must examine the system
	$$\psi = P_h = 0 \text{ and } \lambda \nabla \psi = \mu \nabla P_h$$
	for some $(\lambda,\mu)\neq (0,0).$ Clearly $\lambda \neq 0$ as $P_h$ is smooth. Let us suppose that $\mu=0.$ We consider possible singular points of the form $[r:s:t:0].$ The equations $(\partial \psi/\partial x)(r,s)=(\partial \psi/\partial z_1)(r,s)=0$ yield $r=s=0,$ and then the equation $P_h(0,t,0)=0$ yields $t=0$, a contradiction. Thus we may look for singular points of the form $[r:s:t:1].$ The equation $\partial \psi/\partial x=0$ becomes $f'(x,(\tau-Mx)N^{-1})=0$ and, as above, the equation $\partial \psi/\partial z_1=0$ constrains $s$ to at most $O(1)$ values. In this case, there exists a valid $r$ if and only if the discriminant
	\begin{equation*}
	\mathrm{Disc}_x\bigg(f(x,(\tau-Mx)N^{-1}) - (2ab)^{d-1}hg\bigg(\frac{s+h}{2a},\frac{s-h}{2b}\bigg) -(2ab)^{d-1}k\bigg)
	\end{equation*}
	vanishes. This is the same discriminant as defined in the proof of Lemma~\ref{lemma:t1case} (see~(\ref{eq:discforKh})), and we showed there that this vanishes for at most $O(1)$ values of $\tau.$ Hence the claim holds in this case.
	\\ \\
	Thus we may suppose that $\lambda \mu \neq 0.$ Our equations become
	\begin{align*}
\frac{\partial{\psi}}{\partial{x}} &= 0, \\
\lambda \frac{\partial{\psi}}{\partial{z_1}} &= \mu\frac{\partial{P_h}}{\partial{z_1}},   \\
\frac{\partial{P_h}}{\partial{z_2}}  &= 0, \\
\lambda \frac{\partial{\psi}}{\partial{w}} &= \mu\frac{\partial{P_h}}{\partial{w}}.
\end{align*}
	On replacing $\psi(x,z_1,w)$ with $\psi(x,z_2,w)$ and using the symmetry of $P_h(z_1,z_2,w)$ in the first two arguments, we may adjoin to this system the equation $\partial{P_h}/\partial{z_1}  = 0.$ Now, if $[r:s:t:1]$ is a singular point then the equations $(\partial{P_h}/\partial{z_1})(r,s) =(\partial{P_h}/\partial{z_2})(r,s) =0$ together with Euler's identity and the equation $P_h(s,t,1)=0$ yield $(\partial P_h/\partial w)(r,s)=0.$ This is a contradiction as $P_h$ is smooth. Thus we may restrict ourselves to looking at singular points of the form $[r:s:t:0].$ But now our first equation implies that $r=0,$ and similar to the proof of the claim above, the equations $(\partial P_h/\partial z_1)(r,s)=(\partial P_h/\partial z_2)(r,s)=0$ either cannot be satisfied in the case $d=3,$ or imply that $s=t=0$ in the case $d=4.$ In both cases we arrive at a contradiction. 
	
	\item If $f_d(N,-M)\neq 0$ in $\overline{\mathbb{F}}_p,$ then we must instead consider the equations
	\begin{align*}
	\frac{f_d(N,-M)}{N^d} x^{d} &=g_{d-1}(b,a)hz^{d-1}w + (\text{terms involving $w$}).
	\end{align*} 
	and
	\begin{align*}
	P_h(z_1,z_2,w)= 0.
	\end{align*}
	One can proceed exactly as above and conclude there are at most $O(1)$ values of $\tau$ for which this system is singular.
\end{enumerate}
This completes the proof of the claim.
\end{proof}

Now, for the $O(1)$ values of $\tau$ for which the curve $K_h(x,(\tau-Mx)N^{-1},z,1) =P_h(z_1,z_2,1)=0$ is not absolutely irreducible, we will apply the trivial bound $N_{j}(\tau) \ll p^{j}.$ In the complementary case, by Lemma~\ref{lemma:weil}, we obtain 
\begin{equation*}
N_{j}(\tau) = p^{j} + O(p^{j/2})
\end{equation*}
Thus we may take $N_{j}=p^{j}$ and $\kappa=2$ in the statement of Lemma~\ref{lemma:HMOM}, and the result follows.
\end{proof}

We are left to proof Proposition~\ref{lemma:browningexp} in the cases when $t\in\{3,4\}.$ In this regime we only require an upper bound for $\Sigma_t(p;M,N)$ whenever $p\in \mathcal{P}.$ To do this we follow the argument of Browning~\cite{Browning}. By definition, we have 
\begin{align}
\Sigma_t(p;M,N) &=\sum_{\substack{x,y,z_1,\ldots,z_t\,\,(\text{mod}\,\,p) \\ K_h(x,y,z_1,1)=0 \\ (z_i-z_j)P_h(z_i,z_j,1)=0\text{ for } i\neq j}}e\bigg(\frac{Mx+Ny}{p}\bigg).
\end{align}
Let $\sigma({\bf z})$ denote the number of distinct elements in the set $\{z_1,\ldots,z_t\},$ so that $1\leq \sigma({\bf z}) \leq t.$ We may split our sum according to the value of $\sigma({\bf z}).$ The contribution from those $z_1,\ldots,z_t$ for which $\sigma({\bf z})=1$ is $\Sigma_1$, and this event arises in precisely one way. The contribution from those $z_1,\ldots,z_t$ with $\sigma({\bf z})=2$ is $\Sigma_2-\Sigma_1$ by~(\ref{eq:sigma12}) and this event arises in $c_t$ ways, for some appropriate constant $c_t$ depending only on $t$. Let us now consider the contribution from those ${\bf z}$ with $\sigma({\bf z})=3.$ We claim the following. From the definition of $P_h$ it is easy to see the following:
\begin{enumerate}
	\item If $d=3$ we have
$$P_h(x,y,1) = P_h(x,z,1) \iff y=z.$$
	\item If $d=4$ we have
$$P_h(x,y,1) = P_h(x,z,1) \iff y=z\text{ or } x+y+z=g_{d-1}(b,a)^{-1}h^{-1}c(h)$$
for some polynomial $c(h)$ whose coefficients depend on $g,a$ and $b.$
\end{enumerate}
%\\ \\
%{\it Claim.} When $d=3$ we have
%$$P_h(x,y,1) = P_h(x,z,1) \iff y=z,$$
%and when $d=4$ we have
%$$P_h(x,y,1) = P_h(x,z,1) \iff y=z\text{ or } x+y+z=g_{d-1}(b,a)^{-1}h^{-1}c(h)$$
%for some polynomial $c(h)$ whose coefficients depend on $g,a$ and $b.$
%\begin{proof}[Proof (of claim)]
%Recall from the proof of Proposition~\ref{prop:phsmooth} that
%$$P_h(x,y,1) = g_{d-1}(b,a)h(x+y)+c_1(h)$$
%when $d=3$ and 
%$$P_h(x,y,1) = g_{d-1}(b,a)h(x^2+xy+y^2)+c_2(h)(x+y)+c_3(h)$$
%when $d=4,$ for some polynomials $c_i(h)$ whose coefficient depend on $g,a$ and $b.$ From this it is clear that the claim follows in each case.
%\end{proof}
With this, it is clear there is no contribution from the case $\sigma({\bf x}) \geq 3,$ whenever $d=3.$ Hence, we obtain 
\begin{equation}
\Sigma_t(p;M,N) = (1-c_t)\Sigma_1+c_t\Sigma_2
\end{equation}
whenever $t\in \{3,4\}$ and $d=3,$ and Proposition~\ref{lemma:browningexp} follows. Thus we may suppose that $d=4.$ In this case it is clear that there is no contribution from $\sigma({\bf x}) \geq 4.$ This just leaves us to examine the case $\sigma({\bf x})=3.$ This event will arise in $d_t$ ways, for some appropriate constant $d_t$ depending only on $t$. We must examine the exponential sum
\begin{equation}
\sum_{\substack{x,y,z_1,z_2\,\,(\text{mod}\,\,p) \\ (z_1-z_2)(2z_1+z_2-c(h))(z_1+2z_2-c(h))\neq 0 \\ K_h(x,y,z_1,1)=0 \\ P_h(z_1,z_2,1)=0}}e\bigg(\frac{Mx+Ny}{p}\bigg).
\end{equation}
\sloppy As $p\in \mathcal{P}$ and so, in particular, $p\nmid 6 h g_{d-1}(b,a)),$ the polynomials $P_h(z,z,1), P_h(z,c(h)-2z,1)$ and $P_h(z,2^{-1}(c(h)-z),1)$ are non-zero quadratic polynomials in $z.$ It follows that the terms omitted contribute $O(p)$ altogether. In other words, this equals 
\begin{equation}
\sum_{\substack{x,y,z_1,z_2\,\,(\text{mod}\,\,p)  \\ K_h(x,y,z_1,1)=0 \\ P_h(z_1,z_2,1)=0}}e\bigg(\frac{Mx+Ny}{p}\bigg)+O(p).
\end{equation}
This exponential sum is precisely $\Sigma_2-\Sigma_1+O(p)$ by~(\ref{eq:sigma12}) and~(\ref{eq:sigma12next}). Thus, we obtain
\begin{equation}
\Sigma_t(p;M,N) = (1-c_t-d_t)\Sigma_1(p;M,N)+(c_t+d_t)\Sigma_2(p;M,N) + O(p)
\end{equation}
whenever $t\in\{3,4\}$ and $d=4.$ This completes the proof of Proposition~\ref{lemma:browningexp}.

\section{Estimation of exponential sums (II)}

Recall the definition of $\Phi$ given by~(\ref{eq:defphi}). In this section we will estimate the exponential sums
\begin{equation}
\Phi(h;M,N) = \sum_{\substack{x,y\,\,(\text{mod}\,\,h) \\ (2ab)^{d-1}f(x,y)\equiv (2ab)^{d-1}k\,\,(\text{mod}\,\,h)}}e\bigg(\frac{Mx+Ny}{h}\bigg)
\end{equation}
where $h$ is a fixed, positive integer and $M$ and $N$ are integers. We first note the following multiplicativity property of these sums: if $(j,l)=1$ and $j\overline{j}+l\overline{l}=1$ is it not difficult to show that 
\begin{equation}
\Phi(jl;M,N) = \Phi(j;\overline{l}M,\overline{l}N)\Phi(l;\overline{j}M,\overline{j}N).
\end{equation}
Thus is suffices to study $\Phi(p^l;M,N)$ for some prime $p$ and integer $l\geq 1.$ To tackle these exponential sums we will obtain square-root cancellation in the generic case for $l=1,$ and use elementary bounds for higher powers.

We now consider estimating $\Phi(p;M,N).$ This exponential sum will be sensitive to whether or not the curve $\{f(x_1,x_2)=k\} \subset \mathbb{A}_{\mathbb{F}_p}^{2}$ contains a line. In this case, if this line can be parametrised by $Mx+Ny= \tau,$ for some constant $\tau\in \mathbb{F}_p$, we will get zero cancellation and hence the sum will be large. Because we are assuming the curve $f(x,y)=k$ contains no line definable over $\mathbb{Q},$ this should be a rare event. Our aim is to define a non-zero integer $\Delta_f(M,N,k)$ such that whenever this occurs, we must have that $p|\Delta_f.$ This, together with size bounds on $\Delta_f,$ will be enough to control the cases where this exponential sum is large. This is the content of the following lemma. 

\begin{lemma}\label{lemma:linesinphi}
Let $f,k$ be as in the statement of Theorem~\ref{theo:generalmain}. Fix integers $(M,N)\in \mathbb{Z}^{2},$ not both zero. Fix a prime $p$ such that $p\nmid (M,N)$ and $p$ is sufficiently large in terms of $d.$ Let $\tau\in \overline{\mathbb{F}}_p.$ Then there exists a non-zero integer $\Delta_f(M,N,k)$ such that whenever the curve $\{f(x,y)=k\}\subset \mathbb{A}_{\overline{\mathbb{F}}_p}^{2}$ contains the line $Mx+Ny = \tau,$ we must have $p|\Delta_f(M,N,k).$ Moreover, $\Delta_f(M,N,k)$ satisfies the size bound 
\begin{equation}
|\Delta_f(M,N,k)| \ll_{f,d} |k|\max\{|M|, |N|\}^{d^2}.
\end{equation}
\end{lemma}

\begin{proof}
We may suppose WLOG that $N\neq 0.$ (If $M\neq 0$ then an identical case-analysis argument holds with minor adjustments.) With notation as above, let us suppose that the curve $f(x,y)=k$ contains the line $Mx+Ny=\tau$ over $\overline{\mathbb{F}}_p.$ In other words, we have the polynomial identity
\begin{equation}
f(x,(\tau-Mx)N^{-1})=0 
\end{equation}
in $\overline{\mathbb{F}}_p[x].$ By Taylor expansion (see e.g.~(\ref{eq:taylorexpansion})), and using the fact $p\nmid N,$ we see our assumption is that 
%may write 
%\begin{align*}
%f(x,(\tau-Mx)N^{-1}) &= \sum_{j=0}^{d} \bigg[\sum_{i=j}^{d} \frac{(-\tau)^{i-j}}{N^i(i-j)!} [\partial_y^{i-j} f_i](N,-M) \bigg]x^j.
%\end{align*}
the identity 
\begin{equation}\label{eq:lineFp}
\frac{f_d(N,-M)}{N^d}x^d + \bigg[\frac{f_{d-1}(N,-M)}{N^{d-1}} - \tau\frac{(\partial f_d/\partial y)(N,-M)}{N^d} \bigg]x^{d-1}+\ldots=k
\end{equation}
holds in $\overline{\mathbb{F}}_p[x].$ The proof now proceeds by careful case-analysis.

If $f_d(N,-M)\neq 0$ in $\mathbb{Q}$, then for the leading term to vanish we must have $p|f_d(N,-M).$ Thus in this case we define $\Delta_f(M,N,k) := f_d(N,-M)$ and the size bound is easily satisfied.

\sloppy Thus, proceeding, we may suppose that $f_d(N,-M)=0$ in $\mathbb{Q}$. In this case, by Euler's identity, and the fact $f_d$ is assumed to be smooth and $N\neq 0,$ we must have $(\partial f_d/\partial y)(N,-M)\neq 0$ in $\mathbb{Q}.$ 

If $p|(\partial f_d/\partial y)(N,-M)$ then we set $\Delta_f(M,N,k):=(\partial f_d/\partial y)(N,-M)$ and we are done; again the size bound is immediate.

Now let us suppose that $p\nmid(\partial f_d/\partial y)(N,-M).$ The vanishing of the coefficient of $x^{d-1}$ implies that
\begin{equation}\label{eq:tau}
\tau = \frac{N f_{d-1}(N,-M)}{(\partial f_d/\partial y)(N,-M)}\,\,\,\text{in $\overline{\mathbb{F}}_p$}.
\end{equation}
Substituting~(\ref{eq:tau}) back into~(\ref{eq:lineFp}), we obtain
\begin{equation}\label{eq:newlineFp}
\sum_{j=0}^{d} E_j(M,N) x^j = k \,\,\,\text{in $\overline{\mathbb{F}}_p[x]$,}
\end{equation}
where, for each $j\in \{0,\ldots,d\},$ we have defined the rational numbers 
\begin{align} \nonumber
E_j(M,N) &:= \frac{1}{N^j}\sum_{i=j}^{d} \frac{(-1)^{i-j}}{(i-j)!} \bigg(\frac{f_{d-1}(N,-M)}{(\partial f_d/\partial y)(N,-M)}\bigg)^{i-j}\frac{\partial^{i-j}f_i}{\partial y^{i-j}}(N,-M) \\ \nonumber
&= \frac{\sum_{i=j}^{d} \frac{(d-j)!}{(i-j)!}f_{d-1}(N,-M)^{i-j}(\partial f_d/\partial y)(N,-M)^{d-i} (\partial^{i-j} f_i/\partial y^{i-j})(N,-M)}{N^j(\partial f_d/\partial y)(N,-M)^{d-j}(d-j)!} \\ 
&:= \frac{A_j(M,N)}{B_j(M,N)}.
\end{align}
Note that $A_j(M,N)$ and $B_j(M,N)$ are both integers, $B_j(M,N)$ is non-zero and our assumptions imply that $p\nmid B_j(M,N)$ (for every $j$). 
%In other words, with our convention adopted at the beginning of this section, equation~(\ref{eq:newlineFp}) is indeed well-defined in $\overline{\mathbb{F}}_p[x].$

Suppose that $E_j(M,N)\neq 0$ in $\mathbb{Q}$ for some $j\in \{1,\ldots,d-1\}.$ Let $J$ be the maximal such integer. Then, in particular, we must have that $p$ divides $A_J(M,N)$ and in this case we set $\Delta_f(M,N,k) := A_J(M,N).$ Note 
\begin{equation*}
|\Delta_f(M,N,k)| \ll \max_{J\leq i\leq d}\{|M|,|N|\}^{(d-1)(i-j)+(d-2)(d-i)+j} \ll \max\{|M|,|N|\}^{d^2},
\end{equation*}
and so the size bound is satisfied. 

Otherwise, we arrive at the identity 
\begin{equation*}
E_0(M,N) = k \,\,\,\text{in $\overline{\mathbb{F}}_p$.}
\end{equation*}
If $E_0(M,N)=k$ in $\mathbb{Q},$ then, together with all of our assumptions thus far, we arrive at a bonafide identity over $\mathbb{Q}:$
\begin{equation*}
\sum_{j=0}^{d} E_j(M,N) x^j = k \,\,\,\text{in $\mathbb{Q}[x]$.}
\end{equation*}
Translating everything back, this says that $f(x,y)+k$ contains the line 
\begin{equation*}
Mx+Ny = \frac{N f_{d-1}(N,-M)}{(\partial f_d/\partial y)(N,-M)}.
\end{equation*}
This is a rational line, which is a contradiction.

Thus, we must have that $E_0(M,N)\neq k$ in $\mathbb{Q}.$ In this case we must have $p$ divides the numerator of $E_0(M,N)-k,$ and so we define $\Delta_f(M,N,k):=A_0(M,N)-kB_0(M,N).$ As above, we have $|A_0(M,N)| \ll \max\{|M|,|N|\}^{d^2}.$ Note that 
$$|B_0(M,N)| \ll \max\{|M|,|N|\}^{d(d-1)} \ll \max\{|M|,|N|\}^{d^2}.$$
As
$$|\Delta_f(M,N,k)| \leq |A_0(M,N)| + |k| |B_0(M,N)|,$$
the stated bound follows.
\end{proof}

With Lemma~\ref{lemma:linesinphi}, we can prove the following. 
%Recall
%\begin{equation}
%\Phi(p;M,N) = \sum_{\substack{x,y\,\,(\text{mod}\,\,p) \\ (2ab)^{d-1}f(x,y)\equiv (2ab)^{d-1}k\,\,(\text{mod}\,\,p)}}e\bigg(\frac{Mx+Ny}{p}\bigg).
%\end{equation}

\begin{lemma}\label{lemma:primes}
For any integers $M$ and $N$ we have 
\begin{equation}\label{eq:phipbound}
\Phi(p;M,N)\ll p^{1/2}(p,\Delta_f(M,N,k))^{1/2}.
\end{equation}
\end{lemma}
\begin{proof}
We may suppose that $p$ is sufficiently large in terms of $f,a$ and $b$ as otherwise both sides are $O(1).$ We note that 
\begin{align*}
\#\{x,y\in \mathbb{F}_p: (2ab)^{d-1}f(x,y)\equiv (2ab)^{d-1}k\,\,(\text{mod}\,\,p)\} \ll p.
\end{align*}
Thus, by bounding trivially, we may apply the bound $\Phi(p;M,N) \ll p$ whenever $M=N=0, p|(M,N)$ or $p|\Delta_f(M,N,k).$ By inspecting the proof of Lemma~\ref{lemma:linesinphi}, it is clear that $\Delta_f(0,0,k)=0$ and $(M,N)|\Delta_f(M,N,k).$ Thus~(\ref{eq:phipbound}) holds in these cases. Proceeding, we may suppose $M$ and $N$ are not both zero and $p\nmid \Delta_f(M,N,k).$

In this regime, we will obtain square-root cancellation in $\Phi(p;M,N)$ using Lemma~\ref{lemma:HMOM}. For any $j\geq 1$ and $\tau \in \mathbb{F}_{p^j},$ we define 
$$N_{j}(\tau) = \#\{x,y\in \mathbb{F}_{p^{j}}: (2ab)^{d-1}f(x,y) = (2ab)^{d-1}k\text{ and } Mx+Ny = \tau\}.$$
We may suppose WLOG that $p\nmid N.$ We can write
$$N_{j}(\tau) = \#\{x\in \mathbb{F}_{p^{j}}: (2ab)^{d-1}f(x,(\tau-Mx)N^{-1}) = (2ab)^{d-1}k\}.$$
Since we are assuming $p\nmid \Delta_f(M,N,k),$ it follows from Lemma~\ref{lemma:linesinphi} that the curve $f(x,y)=k$ contains no lines over $\mathbb{F}_{p^j}.$ Hence $f(x,(\tau-Mx)N^{-1})-k$ is never the zero polynomial in $\mathbb{F}_{p^j}[x]$, and we can bound $N_j(\tau) \ll 1.$ Thus we may take $N_{j}=0$ and $\kappa=1$ in the statement of Lemma~\ref{lemma:HMOM}, and the result follows. 
\end{proof}

We now turn our attention to bounding $\Phi(p^l;M,N)$ when $l\geq2.$ To do this we will bound trivially and forego any cancellation in our exponential sum.  In other words, we bound
\begin{equation}
\Phi(p^l;M,N) \leq \#\{x,y\in \mathbb{Z}/p^l\mathbb{Z}: (2ab)^{d-1}f(x,y) \equiv (2ab)^{d-1}k\,\,(\text{mod}\,\,p^l)\}
\end{equation}
and aim to estimate the count on the RHS. To do this we need some understanding of the number of solutions to polynomial congruences over the finite rings $\mathbb{Z}/p^l\mathbb{Z}$.

To this end, let $Q\in \mathbb{Z}[x]$ be a polynomial of degree $d$ with leading coefficient $a_d.$ We will make the dependencies on implied constants explicit for the following estimates involving $Q$. We are interested in general bounds for the count
 \begin{equation}
\#\{x \in \mathbb{Z}/p^l\mathbb{Z}: Q(x)\equiv 0\,\,(\text{mod}\,\,p^l)\}.
\end{equation}
As a first estimate, whenever $p\nmid \mathrm{cont}(Q),$ we have
 \begin{equation}\label{eq:firstestimate}
\#\{x \in \mathbb{Z}/p^l\mathbb{Z}: Q(x)\equiv 0\,\,(\text{mod}\,\,p^l)\} \ll_{d} p^{l-1}.
\end{equation}
This is easily proven by induction; there are $O_d(1)$ roots when $l=1$ and, when $l\geq 2,$ any element of $\mathbb{Z}/p^{l-1}\mathbb{Z}$ has exactly $p$ lifts to an element of $\mathbb{Z}/p^l\mathbb{Z}.$ 

By using $p$-adic arithmetic, we can improve on this bound for large $l$.
 
\begin{lemma}
With notation as above, for any prime $p$ and integer $l\geq 1,$ we have 
\begin{equation}\label{eq:padicbound}
\#\{x\in \mathbb{Z}/p^{l}\mathbb{Z}: Q(x)\equiv 0\,\,(\text{mod}\,\,p^l)\} \ll_{d} p^{l-l/d+v_p(a_d)/d},
\end{equation}
where $v_p(a_d)$ denotes the $p$-adic valuation of the non-zero integer $a_d.$
\end{lemma}

\begin{proof}
We will prove this by using the fact any root must be ``$p$-adically close" to one of the $d$ (not necessarily distinct) roots of $Q$ in an algebraic closure of $\mathbb{Q}_p,$ the $p$-adic numbers. To this end, we define the following (standard) notation. We let $|\cdot|_p$ denote the usual norm in $\overline{\mathbb{Q}}_p$ and $\mu$ denote the Haar measure. For any $\beta \in \overline{\mathbb{Q}}_p$ and $m\in \mathbb{Z}$ we define 
$$B(\beta,p^{m}) := \{y\in \overline{\mathbb{Q}}_p: |y-\beta|_p \leq p^m\},$$
i.e. the (closed) ball of radius $p^m$ centered at $\beta,$ so that $\mu(B(\beta,p^m))=p^m.$
The estimate~(\ref{eq:padicbound}) is true for $l=1$ since in this case we can use the bound $O_d(1).$ 

Proceeding, we fix an integer $l\geq 2$ and roots $\beta_1,\ldots,\beta_d$ of $Q$ in an algebraic closure $\overline{\mathbb{Q}}_p$. In $\overline{\mathbb{Q}}_p[t]$ we have the factorisation
\begin{equation*}
Q(t) = a_d\prod_{i=1}^{d}(t-\beta_i).
\end{equation*}
We can partition the set we are interested as follows:
\begin{equation*}
\bigcup_{i=1}^{d} \#\{x\in \mathbb{Z}/p^{l}\mathbb{Z}: Q(x)\equiv 0\,\,(\text{mod}\,\,p^l)\text{ and } |x-\beta_i|_p\text{ minimal}\}.
\end{equation*}
Suppose we are looking at the set corresponding to the root $\beta.$ For an element $x\in \mathbb{Z}/p^l\mathbb{Z}$ to be counted we must have
\begin{equation*}
|a_d|_p |x-\beta|_p^{d} \leq |a_d|_p \prod_{i=1}^{d}|x-\beta_i|_p  = |Q(x)|_p \leq p^{-l}.
\end{equation*}
Suppose $m=v_p(a_d)$ so that $|a_d|_p = p^{-m}.$ The line above equivalently says that $x \in B(\beta,p^{-(l-m)/d})$. By reduction, we see that $x$ must take the form $x = \alpha + \lambda p$ for some root $\alpha\in\mathbb{Z}/p\mathbb{Z}$ and $\lambda \in\mathbb{Z}/p^{l-1}\mathbb{Z}.$ There are $O_d(1)$ choices for $\alpha.$ Suppose that $\alpha + \lambda p\in\overline{B}(\beta,p^{-(l-m)/d})$ for $\lambda\in \Lambda \subset \mathbb{Z}/p^{l-1}\mathbb{Z}.$ Our aim is therefore to bound $|\Lambda|.$ We claim the following:
\\ \\
{\it Claim:} If $\lambda \neq \lambda'\in \mathbb{Z}/p^{l-1}\mathbb{Z}$ then the balls $B_{\lambda} := B(\alpha+\lambda p, p^{-l})$ are disjoint.
\begin{proof}[Proof (of claim)]
For any distinct elements $\lambda \neq \lambda'\in \mathbb{Z}/p^{l-1}\mathbb{Z}$ we have 
$$|\lambda-\lambda'|_p \geq p^{-(l-2)}.$$
Now, if $x \in \overline{B}_{\lambda}\cap \overline{B}_{\lambda'}$ we have
\begin{align*}
p^{-(l-1)}&\leq |(\alpha+\lambda p) - (\alpha+\lambda'p)|_p  \\
&=|(\alpha+\lambda p)-x - (\alpha+\lambda'p-x)|_p \\
&\leq \max\{|\alpha+\lambda p-x|_p,  |\alpha+\lambda'p-x|_p\} \\
&\leq p^{-l},
\end{align*}
a contradiction. Here we have used the fact $|\cdot|_p$ is non-Archimedean.
\end{proof}

With this claim proven we can complete the proof. Fix $\alpha$. As $|\cdot|_p$ is non-Archimedean, whenever $\alpha + \lambda p\in B(\beta,p^{-(l-m)/d})$ we must actually have 
\begin{equation*}
B(\alpha + \lambda p,p^{-(l-m)/d}) = B(\beta,p^{-(l-m)/d}).
\end{equation*}
Hence we are assuming that 
\begin{equation*}
\bigcup_{\lambda \in \Lambda} B(\alpha+\lambda p, p^{-l}) \subseteq \bigcup_{\lambda \in \Lambda} B(\alpha+\lambda p, p^{-(l-m)/d}) = B(\beta,p^{-(l-m)/d}).
\end{equation*}
From the claim, these sets on the LHS are disjoint. Using the fact $\mu$ is a measure, it follows that 
\begin{align*}
\mu \bigg(\bigcup_{\lambda \in \Lambda}B(\alpha+\lambda p^l,p^{-l})\bigg) &= \sum_{\lambda \in \Lambda}\mu(B(\alpha+\lambda p^l,p^{-l})) = |\Lambda| p^{-l}.
\end{align*}
On the other hand 
\begin{equation*}
\mu(B(\beta,p^{-(l-m)/d})) = \mu(B(\beta,p^{- \left \lceil (l-m)/d \right \rceil }))  =  p^{-\left \lceil (l-m)/d \right \rceil}.
\end{equation*}
Putting these facts together, and using monotonicity of $\mu,$ we obtain 
\begin{align*}
|\Lambda| &\leq p^{l-\left \lceil (l-m)/d \right \rceil} \leq p^{l-l/d+m/d}.
\end{align*}
We are done as there are only $O_d(1)$ possibilities for $\alpha$ and $\beta$. 
\end{proof}

These results can easily be extended to the case $p^m||\mathrm{cont}(Q).$ 
\begin{corollary}\label{cor:expsumcor}
Suppose $p^m||\mathrm{cont}(Q).$ Then
 \begin{equation}\label{eq:polycong}
\#\{x \in \mathbb{Z}/p^l\mathbb{Z}: Q(x)\equiv 0\,\,(\text{mod}\,\,p^l)\} \ll_{d,a_d}  
\begin{cases}
p^l\,\,&\text{if $l\leq m,$} \\
p^{\min\{l-1,l-l/d\}})\,\,&\text{if $l>m$.}
\end{cases}
\end{equation}
\end{corollary}
\begin{proof}
The statement is clear when $l\leq m$ and so we suppose that $l>m.$ Write $Q(x) = p^m R(x)$ where $\mathrm{cont}(R)=1.$ Then we have
 \begin{align*}
 \#\{x \in \mathbb{Z}/p^l\mathbb{Z}: Q(x)\equiv 0\,\,(\text{mod}\,\,p^l)\}  
 &= p^N\cdot\#\{y \in \mathbb{Z}/p^{l-m}\mathbb{Z}: R(y)\equiv 0\,\,(\text{mod}\,\,p^{l-m})\}.
 \end{align*}
 Hence, using~(\ref{eq:firstestimate}) and~(\ref{eq:padicbound}), we can bound this by 
 $$\ll_d p^m\cdot p^{\min\{l-m-1, l-m - (l-m)/d+v_p(a_d p^{-m})/d\}} = p^{\min\{l-1, l - l/d+v_p(a_d)/d\}}.$$
 The result follows upon noting the $p^{v_p(a_d)/d}$ factor can be absorbed into the constant term, which we allow to depend on $a_d.$
 \end{proof}

The following estimate will also be convenient for us: for any $A,B\in \mathbb{Z}$ such that $A\neq 0$ and $p\nmid A,$ one has 
\begin{equation}\label{eq:linearcong}
\#\{x \in \mathbb{Z}/p^l\mathbb{Z}: p^m|Ax+B\} \leq 
\begin{cases}
p^{l-m}\,\,&\text{if $m\leq l,$} \\
1 \,\,&\text{if $m>l$.}
\end{cases}
 \end{equation}
To proceed we require the following lemma.
 \begin{lemma}\label{lemma:nonzerodisc}
Let $f\in \mathbb{Z}[x,y]$ be such that the projective variety $\{f_d(x,y)=0\}\subset\mathbb{P}_{\mathbb{Q}}^{1}$ has no non-constant repeated factors. Then the polynomial 
$$\mathrm{Disc}_t((2ab)^{d-1}f(x,t)-(2ab)^{d-1}k)$$ 
is not identically zero, and the leading coefficient is independent of $k$.
 \end{lemma}
\begin{proof}
We wish to show a particular resultant polynomial is not identically zero. We will follow the strategy adopted in Lemma~\ref{lemma:singvals}: once again, we will extract the leading coefficient and show this is non-zero. Since
\begin{align*}
f(x,xt) &= \sum_{i=0}^{d} f_i(x,xt) = \sum_{i=0}^{d} x^i f_i(1,t),
\end{align*}
it follows that 
$$\lim_{x\rightarrow \infty} \frac{(2ab)^{d-1}f(x,xt)-(2ab)^{d-1}k}{x^d} = (2ab)^{d-1}f_d(1,t).$$
By a similar argument to the proof of Lemma~\ref{lemma:singvals}, it follows that our discriminant polynomial has the leading coefficient 
$$\mathrm{Disc}_t((2ab)^{d-1}f_d(1,t)).$$
Now this discriminant is non-zero by our assumption on $f_d(x,y),$ and it is clearly independent of $k.$
\end{proof}

Let us write $f(x,y)=\sum_{i+j\leq d}c_{i,j}x^iy^j,$ so that 
\begin{equation}
f_{d}(x,y) = c_{d,0}x^d+c_{d-1,1}x^{d-1}y+\ldots+c_{1,d-1}x y^{d-1}+c_{0,d}y^d.
\end{equation}
If both $c_{d,0}=c_{d-1.1}=0$ then $f_d(x,y)$ will contain a a square factor. Thus with our assumptions at least one of $c_{d-1,1}$ or $c_{d,0}$ is non-zero. (Similarly for $c_{0,d}$ and $c_{1,d-1}.$) We will make use of this fact in what follows. Our main result is the following. We recall our convention that implied constants may depend on $f,g,a$ and $b$ without specifying so, and this includes dependencies on the coefficients of $f$ and $g$ as well as dependencies on the degree $d$.

\begin{lemma}\label{lemma:higherpowers}
Fix $l\geq 2.$ We have
$$
\Phi(p^l;0,0) \ll
\begin{cases}
p^{2l-2}\,\,&\text{if $2\leq l\leq d,$} \\
p^{2l-l/d-1}\,\,&\text{if $l\geq d.$}
\end{cases}
$$
\end{lemma}
\begin{proof}
For this proof it will also be helpful to define the polynomial
\begin{equation}
F(x,y) := (2ab)^{d-1}f(x,y)-(2ab)^{d-1}k.
\end{equation}
We may assume that $p$ is sufficiently large in terms of $f,a$ and $b$ as otherwise the result holds trivially. In particular, in view of Lemma~\ref{lemma:nonzerodisc}, in this regime we may assume the polynomial $\mathrm{disc}_t(F(x,t))$ has content coprime to $p.$ By Hensel's lemma, if $x\in \mathbb{Z}/p^{l}\mathbb{Z}$ is such that $\mathrm{disc}_t(F(x,t))\not\equiv 0\,\,(\text{mod}\,\,p),$ then any of the $\leq p^l$ choices for $x$ will lead to at most $O(1)$ solutions in $y$. Hence we have
\begin{equation}\label{eq:auxsumphi}
\Phi(p^l;0,0) = \sum_{\substack{x,y\in \mathbb{Z}/p^{l}\mathbb{Z} \\ F(x,y)\equiv 0\,\,(\text{mod}\,\,p^{l}) \\ \mathrm{disc}_t(F(x,t))\equiv 0\,\,(\text{mod}\,\,p) }}1+O(p^l).
\end{equation}
Let us take the sum over $x$ on the outside. The last constraint restricts $x$ to $O(1)$ values modulo $p$ and hence $O(p^{l-1})$ values modulo $p^l.$ We would like to use the estimates proved above to tackle the inner sum over $y$. To do this, we need some understanding of how $p$ divides the content of $F(x,y)$ as a polynomial in $y.$ We write this as $\mathrm{cont}_y(F(x,y)).$ 

If $f_d(0,1)=c_{0,d}\neq 0$ we have 
\begin{equation}
F(x,y) = (2ab)^{d-1}c_{0,d}y^d+(\text{lower order terms in $y$}).
\end{equation}
In the regime under consideration, $p\nmid c_{0,d}$, and so we get an overall bound  
\begin{equation}
\Phi(p^l;0,0)  \ll p^l+p^{l-1+\min\{l-1,l-l/d\}}
\end{equation}
by Corollary~\ref{cor:expsumcor}. This agrees with the result stated.

Now let us suppose that $f_d(0,1)=c_{0,d}=0.$ From the remarks above, we must therefore have $c_{1,d-1}\neq 0.$ In particular, we are assuming $p\nmid c_{1,d-1}.$ It follows that  
\begin{equation}
F(x,y) = (2ab)^{d-1}(c_{1,d-1}x+c_{0,d-1})y^{d-1}+(\text{lower order terms in $y$})
\end{equation}
for some $c_{0,d-1}\in \mathbb{Z}$. We now split the sum in~(\ref{eq:auxsumphi}) according to the power $p^m$ which divides $\mathrm{cont}_y(F(x,y)).$ Write this sum as
\begin{equation}\label{eq:sumoverk}
\sum_{m} \sum_{\substack{x\in \mathbb{Z}/p^{l}\mathbb{Z} \\ \mathrm{disc}_t(F(x,t))\equiv 0\,\,(\text{mod}\,\,p) \\ p^m||\mathrm{cont}_y(F(x,y))}}\sum_{\substack{y\in \mathbb{Z}/p^{l}\mathbb{Z} \\ F(x,y)\equiv 0\,\,(\text{mod}\,\,p^{l}) }}1. 
\end{equation}
The term with $m=0$ contributes $\ll_d p^{l-1+\min\{l-1,l-l/d\}}$ by Corollary~\ref{cor:expsumcor}. To tackle the remaining sums we will majorise by replacing the condition $p^m||\mathrm{cont}_y(F(x,y))$ with the simpler condition $p^m | (2ab)^{d-1}(c_{1,d-1}x+c_{0,d-1}).$ The terms with $m\geq l+1$ contribute 
\begin{equation}
\leq p^l\sum_{\substack{x\in \mathbb{Z}/p^l\mathbb{Z} \\ \mathrm{disc}_t(F(x,t)) \equiv 0\,\,(\text{mod}\,\,p) \\ p^{l+1}|\mathrm{cont}_y(F(x,y)) }}1  \leq p^l\sum_{\substack{x\in \mathbb{Z}/p^l\mathbb{Z} \\ p^{l+1}|(2ab)^{d-1}(c_{1,d-1}x+c_{0,d-1}) }}1 \leq p^l
\end{equation}
by~(\ref{eq:linearcong}). To finish we need to estimate the contribution to~(\ref{eq:sumoverk}) from when $m\in\{1,\ldots,l\}.$ To do this, we split into two cases.
\begin{enumerate}
	\item Suppose that $l\leq d.$ Fix $1\leq m\leq l.$ By Corollary~\ref{cor:expsumcor} we may bound the sum over $y$ by $O(p^{l-1}).$ We therefore get a contribution
	\begin{align*}
	&\ll p^{l-1} \sum_{m=1}^{l} \sum_{\substack{x\in \mathbb{Z}/p^l\mathbb{Z} \\ p^m| (2ab)^{d-1}(c_{1,d-1}x+c_{0,d-1})}}1
	\end{align*}
	Using~(\ref{eq:linearcong}) we may bound this by
	\begin{align*}
	&\ll p^{l-1} \sum_{m=1}^{l} p^{l-m} \ll p^{2l-2}.
	\end{align*}
	Putting everything together, in this case we get 
	\begin{align}
	\Phi(p^l;0,0) \ll p^{2l-2}+p^{l-1+\min\{l-1,l-l/d\}}+p^l \ll p^{2l-2}.
	\end{align}
	This agrees with the stated result.
	\\ 
	\item Now suppose that $l>d.$ In this case we split our sum at $m=l-d.$ 
	\\
	\begin{enumerate}
		\item Fix $1\leq m\leq l-d.$ We can bound the sum over $y$ by $O(p^{l-l/d})$ using Corollary~\ref{cor:expsumcor}. These terms therefore contribute
		\begin{align*}
	\ll p^{l-l/d} \sum_{m=1}^{l-d} \sum_{\substack{x\in \mathbb{Z}/p^l\mathbb{Z} \\ p^m|c_{1,d-1}x+c_{0,d-1}  }} 1.
	\end{align*}
	Using~(\ref{eq:linearcong}) we can bound this by
	\begin{align*}
	&\ll p^{l-l/d} \sum_{m=1}^{l-d}p^{l-m} \ll p^{2l-l/d-1}.
	\end{align*}
		\item Fix $l-d<m\leq l.$ We can bound the sum over $y$ by $O(p^{l-1})$ using Corollary~\ref{cor:expsumcor}. 
		These terms therefore contribute
	\begin{align*}
	&\ll p^{l-1} \sum_{m=l-d+1}^{l} \sum_{\substack{x\in \mathbb{Z}/p^l\mathbb{Z} \\ p^m| c_{1,d-1}x+c_{0,d-1} }} 1.
	\end{align*}
	Again, using equation~(\ref{eq:linearcong}) we see the overall contribution in this case is
	\begin{align*}
	\ll p^{l-1} \sum_{m=l-d+1}^{l} p^{l-m} \ll p^{l+d-2}.
	\end{align*}
	\end{enumerate}
	Now, since $l>d$ the first term dominates. Putting everything together, we conclude that 
	\begin{equation}
	\Phi(p^l;0,0)\ll p^{2l-l/d-1}+p^{l-1+\min\{l-1,l-l/d\}}+p^l \ll p^{2l-l/d-1}.
	\end{equation}
	This agrees with the result stated. 
\end{enumerate}
\end{proof}

Finally, we have the following.

\begin{prop}\label{prop:phisum}
Fix $\epsilon>0.$ For any integers $M,N$ we have
\begin{equation}
\sum_{0<|h|\ll B} \frac{\Phi(h;M,N)}{|h|^{2-1/d+\epsilon}} \ll_{\epsilon} \Delta_f(M,N,k)^{\epsilon} B^{\epsilon}.
\end{equation}
\end{prop}
\begin{proof}
Fix $0<\delta<1/d.$ By multiplicativity, and the fact $\Phi(-h;M,N)=\Phi(h;M,N),$ we can bound our sum by
\begin{equation*}
\sum_{0<|h|\ll B} \frac{\Phi(h;M,N)}{|h|^{2-\delta}} \ll \prod_{p\ll B}\bigg(1+\frac{\Phi(p;M,N)}{p^{2-\delta}}+\sum_{l=2}^{\infty}\frac{\Phi(p^l;0,0)}{p^{l(2-\delta)}}\bigg).
\end{equation*}
We first deal with the contribution from higher powers. Using the bounds established in Lemma~\ref{lemma:higherpowers}, this contribution is bounded by
\begin{align*}
&\ll \frac{1}{p^2} \sum_{l=2}^{d}p^{l\delta } + \frac{1}{p}\sum_{l=d+1}^{\infty}\frac{1}{p^{l/d-\delta l}} \ll_{\delta} \frac{1}{p^{2-d\delta}} + \frac{1}{p^{2+1/d-(d+1)\delta}}  \ll_{\delta} \frac{1}{p},
\end{align*}
using the fact $\delta<1/d.$ Now, fix positive constants $A_i>0$ which may depend on $f,a,b$ and $\delta.$ We recall Lemma~\ref{lemma:primes}, which says that 
\begin{equation*}
|\Phi(p;M,N)| \ll p^{1/2}(p,\Delta_f(M,N,k))^{1/2}.
\end{equation*}
We see the contribution from those primes dividing $\Delta_f(M,N,k)$ is bounded by
\begin{equation*}
\prod_{p|\Delta_f(M,N,k)} \bigg(1+\frac{A_1}{p^{1-\delta}}+\frac{A_2}{p}\bigg) \ll_{\delta} \tau(\Delta_f(M,N,k)) \ll_{\delta,\epsilon} \Delta_f(M,N,k)^{\epsilon}.
\end{equation*}
The contribution from the remaining terms is bounded by
\begin{align*} 
\prod_{p\ll B}\bigg(1+\frac{A_3}{p^{3/2-\delta}}+\frac{A_4}{p}\bigg) &\ll_{\delta}  \prod_{p\ll B}\bigg(1+\frac{A_5}{p}\bigg) \\
&\ll_{\delta} \prod_{p\ll B}\bigg(1+\frac{1}{p}\bigg)^{A_5} \ll_{\delta,\epsilon} B^{\epsilon}
\end{align*}
by Mertens' estimate. The proof of the proposition is completed upon taking $\delta=1/d-\epsilon.$
\end{proof}

%%%%%%%%%%%%%%
%%%%%%%%%%%%%%
%%%%%%%%%%%%%%
%%%%%%%%%%%%%%
%%%%%%%%%%%%%%

\section{Final estimates}
Let us recall our work so far. Putting together equations~(\ref{eq:removesing}) and~(\ref{eq:applysieve}) we arrive at the bound 
\begin{equation}
M_{f,g}(B;k) \ll_{\epsilon} \sum_{\substack{0<|h| \ll_{a,b} B  \\ K_{h},P_{h}\text{ smooth}}}\,\,\,\,\,\frac{1}{(\#\mathcal{P})^{2}}\sum_{p,q\in\mathcal{P}}\bigg|\sum_{i,j\in\{0,1,2\}}c_{i,j}(\alpha)S_{i,j}(p,q)\bigg|+B^{3/2+\epsilon}
\end{equation}
where
\begin{align}
S_{i,j}(p,q)  &= \frac{1}{(pqh)^{2}} \sum_{-pq|h|/2<m,n\leq pq|h|/2} \Gamma(B,m)\Gamma(B,n)\Psi_{i,j}(m,n).
\end{align}
Using Lemma~\ref{lemma:multi} we were able to decompose the exponential sums $\Psi_{i,j}$ into the simpler exponential sums $\Sigma_t$ and $\Phi$ which we were able to estimate individually. In this section we aim to bring everything together. 

Firstly, we isolate the main term contribution to the above sum. To this end, let $L_{i,j}(p,q)$ denote the the contribution from the term $(m,n)=(0,0)$. Note that $\Gamma(B,0)=\left\lfloor B \right \rfloor,$ and so, with the estimate $ \left\lfloor B \right \rfloor = B+O(1),$ we obtain
\begin{equation}\label{eq:mainterm}
L_{i,j}(p,q) =  \frac{B^2\Psi_{i,j}(0,0)}{(pqh)^{2}}+O\bigg(\frac{B\Psi_{i,j}(0,0)}{(pqh)^{2}}\bigg).
\end{equation}
Let us concentrate our attention on this first term. If $p\neq q,$ then for $i,j\in \{0,1,2\},$ we have
\begin{align} \nonumber
\Psi_{i,j}(0,0) &= \Sigma_i(p;0,0)\Sigma_j(q;0,0)\Phi(h;0,0) \\
&= \max\{1,i\} \max\{1,j\}[p^2+O(p)][q^2+O(q)]\Phi(h;0,0)
\end{align}
using Lemma~\ref{lemma:multi} for the decomposition of $\Psi_{i,j}$ and Lemma~\ref{lemma:browningexp} for the asymptotics for $\Sigma_t.$ Recalling the definition of the constants $c_{i,j}(\alpha)$ in Proposition~\ref{prop:sieve}, we obtain an expression
\begin{align}
\sum_{i,j\in\{0,1,2\}}c_{i,j}(\alpha)L_{i,j}(p,q) &= \frac{B^2\Phi(h;0,0)}{h^2}(\alpha-1)^2+O\bigg(\frac{B^2\Phi(h;0,0)}{\min\{p,q\}h^2}\bigg).
\end{align} 
Thus we may take $\alpha=1$ to eliminate the main term. In this case, including the error present in equation~(\ref{eq:mainterm}) and using the bound $\Psi_{i,j}(0,0) \ll p^2q^2\Phi(h;0,0)$ which follows from our work so far, we obtain
\begin{equation}
\sum_{i,j\in\{0,1,2\}}c_{i,j}(1)L_{i,j}(p,q) \ll \frac{B^2\Phi(h;0,0)}{\min\{p,q\}h^2}+\frac{B\Phi(h;0,0)}{h^{2}}
\end{equation}
for the total contribution to the main term from terms with $p\neq q.$ We see the first term dominates in our range of $p$ and $q$. If $p=q$ then we bound trivially, to obtain 
\begin{align}
\sum_{i,j\in\{0,1,2\}}c_{i,j}(1)L_{i,j}(p,p)  &\ll \frac{B^2\Phi(h;0,0)}{h^2}.
\end{align}
Now note that, from Lemma~\ref{lemma:primes} and Lemma~\ref{lemma:higherpowers}, it follows that we can bound  $\Phi(h;0,0) \ll A^{\omega(h)}h^2/\mathrm{rad}(|h|),$ where $A>0$ is a constant and $\mathrm{rad}(|h|)$ denotes the radical of the non-zero integer $h$. Thus 
\begin{align} \nonumber
\sum_{0<|h|\ll B}\frac{\Phi(h;0,0)}{h^2} & \ll \sum_{0<|h|\ll B}\frac{A^{\omega(h)}}{\mathrm{rad}(|h|)} \\\nonumber
&\ll_{\epsilon} B^{\epsilon}\sum_{0<|h|\ll B}\frac{A^{\omega(h)}}{h^{\epsilon}\mathrm{rad}(|h|)} \\\nonumber
&\ll_{\epsilon} B^{\epsilon}\prod_{p\ll B}\bigg(1+\frac{A}{p}\bigg[\frac{1}{p^{\epsilon}}+\frac{1}{p^{2\epsilon}}+\ldots\bigg]\bigg) \\
&\ll_{\epsilon} B^{\epsilon}\prod_{p\ll B}\bigg(1+\frac{A}{p^{1+\epsilon/2}}\bigg) \ll_{\epsilon} B^{\epsilon}.
\end{align}
Thus, the total contribution from the main term, using the asymptotic $\#\mathcal{P}\sim Q/\log{Q},$ is
\begin{align} \nonumber
 \frac{B^2\log^2{Q}}{Q^2}\sum_{0<|h|\ll B}\frac{\Phi(h;0,0)}{h^2}&\bigg[\sum_{\substack{p,q\in\mathcal{P} \\ p\neq q }}\frac{1}{\min\{p,q\}}+\sum_{p\in \mathcal{P}}1\bigg] \\ \nonumber
 &\ll \frac{B^2\log{Q}\log\log{Q}}{Q}\sum_{0<|h|\ll B}\frac{\Phi(h;0,0)}{h^2} \\ \label{eq:maintermcont}
 &\ll_{\epsilon}  \frac{B^{2+\epsilon}}{Q}. 
\end{align}
%\begin{align*}
%\frac{\log^2{Q}}{Q^2}\sum_{0<|h|\ll B}\sum_{p,q\in\mathcal{P}}\bigg|\sum_{i,j\in\{0,1,2\}}c_{i,j}(1)M_{i,j}\bigg| &\ll \frac{B^2\log^2{Q}}{Q^2}\sum_{0<|h|\ll B}\frac{\Phi(h;0,0)}{h^2}\bigg[\sum_{\substack{p,q\in\mathcal{P} \\ p\neq q }}\frac{1}{\min\{p,q\}}+\sum_{p\in \mathcal{P}}1\bigg] \\ 
%&\ll \frac{B^2\log{Q}}{Q}\sum_{0<|h|\ll B}\frac{\Phi(h;0,0)}{h^2} \ll_{\epsilon}  \frac{B^{2+\epsilon}}{Q}. 
%\end{align*}
%\begin{equation}
%\Phi(h;0,0) \ll \frac{A^{\omega(h)}h^2}{\mathrm{rad}(|h|)},
%\end{equation}
%for some positive $A>0.$ Here $\mathrm{rad}(|h|)$ denotes the radical of the non-zero integer $|h|.$ Hence
%\begin{align}
%\sum_{0<|h|\ll B}\frac{\Phi(h;0,0)}{h^2} \ll_{\epsilon}\sum_{0<|h|\ll B}   \frac{A^{\omega(h)}}{\mathrm{rad}(|h|)}
%\end{align}
%In this range of variables we have $A^{\omega(h)} \ll_{\epsilon} B^{\epsilon}.$ Finally the estimate 
%\begin{equation}
%\sum_{0<|h|\ll B} \frac{1}{\mathrm{rad}(|h|)} \ll_{\epsilon} B^{\epsilon}
%\end{equation}
%is standard. 
We now focus on the contribution from the remaining terms. From now on we put $T_{i,j}= S_{i,j}-L_{i,j}.$ It follows that
\begin{align}
T_{i,j}(p,q) &\ll  \frac{1}{(pqh)^2} \sum_{\substack{-pq|h|/2 < m,n \leq pq|h|/2 \\ (m,n)\neq (0,0)}}\min\bigg\{B, \frac{pq |h|}{|m|}\bigg\}\min\bigg\{B, \frac{pq |h|}{|n|}\bigg\}|\Psi_{i,j}(m,n)|.
\end{align}
We would like to find a pointwise estimate for the exponential sums $\Psi_{i,j}.$ There are two regimes to consider.
\begin{enumerate}
	\item If $p\neq q,$ then recalling Lemma~\ref{lemma:multi} and Lemma~\ref{lemma:browningexp} we obtain
	$$|\Psi_{i,j}(m,n)| \ll pq(pq,m,n)|\Phi(h;\overline{pq}m,\overline{pq}n)|.$$
	Here $\overline{pq}$ is the multiplicative inverse of $pq$ modulo $|h|.$
	\item If $p=q$, then again recalling the results above we have
$$|\Psi_{i,j}(m,n)| \ll 1_{p|(m,n)} p^3(p,m/p,n/p) \Phi(h; \overline{p}m/p, \overline{p}n/p).$$
Here $\overline{p}$ is the multiplicative inverse of $p$ modulo $|h|.$ 
\end{enumerate}
Now if $p=q$ then $1_{p|(m,n)} p^3(p,m/p,n/p) \leq pq(pq,m,n).$
%Also, if $p=q$ and $p|(m,n)$ we have 
%$$\overline{pq}m = \overline{p}^2\cdot p \cdot m/p = \overline{p} \cdot m/p$$
%and sim. $\overline{pq}n = \overline{p} \cdot n/p.$
We conclude that the first bound holds in all cases. Thus we may write
\begin{align}
T_{i,j} &\ll \frac{1}{pqh^2} \sum_{\substack{-pq|h|/2 < m,n \leq pq|h|/2 \\ (m,n)\neq (0,0)}}\min\bigg\{B, \frac{pq |h|}{|m|}\bigg\}\min\bigg\{B, \frac{pq |h|}{|n|}\bigg\}(pq,m,n)|\Phi(h;\overline{pq}m,\overline{pq}n)|.
\end{align}
There are now 3 regimes to consider.
\begin{enumerate}
	\item The contribution from when $m=0$ is 
	\begin{align}
	\frac{B}{|h|}  \sum_{\substack{-pq|h|/2 < n\leq pq|h|/2 \\ n\neq 0}}\frac{(pq,n)|\Phi(h;0,\overline{pq}n)|}{|n|}.
	\end{align}	
	\item The contribution from when $n=0$ is 
	\begin{align}
	\frac{B}{|h|}  \sum_{\substack{-pq|h|/2 < m \leq pq|h|/2 \\ m\neq 0}}\frac{(pq,m)|\Phi(h;\overline{pq}m,0)|}{|m|}.
	\end{align}	
	\item The contribution from when $mn\neq 0$ is 
	\begin{align}
  pq\sum_{\substack{-pq|h|/2 < m,n \leq pq|h|/2 \\  mn \neq0}}\frac{(pq,m,n)|\Phi(h;\overline{pq}m,\overline{pq}n)|}{|m||n|}.
	\end{align}
\end{enumerate}
To evaluate these sums we will now bring the sum over $h$ on the inside. Recall Proposition~\ref{prop:phisum}, which states that for any $\epsilon>0$ we have 
\begin{equation}
\sum_{0<|h|\ll B} \frac{\Phi(h;M,N)}{|h|^{2-1/d+\epsilon}} \ll_{\epsilon} \Delta_f(M,N,k)^{\epsilon}B^{\epsilon},
\end{equation}
and the corresponding results which follow by partial summation. From Lemma~\ref{lemma:linesinphi} we have the size bound
\begin{equation}
|\Delta_f(M,N,k)| \leq |k| \max\{|M|,|N|\}^{d^2}.
\end{equation}
We are also assuming $0<|k| \ll_f B^{d+1}$ (see~(\ref{eq:ksize})). In our range of variables we thus have 
\begin{equation}
\Delta_f(\overline{pq}m,\overline{pq}n,k)^{\epsilon} \ll_{\epsilon} B^{\epsilon}
\end{equation}
It follows that the total contribution from the terms $T_{i,j}$ is bounded by
\begin{align}
\frac{B^{2-1/d+\epsilon}\log^{2}{Q}}{Q^{2}}\sum_{\substack{p,q\in \mathcal{P}}} \bigg[\sum_{\substack{-pq|h|/2 < n \leq pq|h|/2 \\ n\neq 0}}\frac{(pq,n)}{|n|}+pq\sum_{\substack{-pq|h|/2 < m,n \leq pq|h|/2 \\  mn \neq0}}\frac{(pq,m,n)}{|m||n|}\bigg].
\end{align}
This in turn is bounded by
\begin{align}\label{eq:errorterms}
\frac{B^{2-1/d+\epsilon}\log^{2}{Q}}{Q^{2}}\sum_{\substack{p,q\in \mathcal{P}}}pq &\ll_{\epsilon} Q^2B^{2-1/d+\epsilon}.
\end{align}
Putting together equation~(\ref{eq:maintermcont}) and equation~(\ref{eq:errorterms}), we get a final bound
\begin{equation}
M_{f,g}(B;k) \ll_{\epsilon} B^{\epsilon}\bigg(Q^2B^{2-1/d}+ \frac{B^{2}}{Q} \bigg) \ll B^{2-1/(3d)+\epsilon},
\end{equation}
where we take $Q=B^{1/(3d)}$ to balance the error terms. The result stated in Theorem~\ref{theo:generalmain} follows upon noting that the exponent here is strictly less than $(2-1/(50d)).$

\end{document}